\documentclass[10pt]{preprint}

\usepackage[margin=3.5cm]{geometry}
\usepackage[utf8x]{inputenc}
\usepackage[english]{babel}
\usepackage[full]{textcomp}
\usepackage[osf]{newtxtext}
\usepackage{amssymb}
\usepackage{amsmath}
\usepackage{amsthm}
\usepackage{breakurl}
\usepackage{mhequ}
\usepackage{booktabs}
\usepackage{tikz}
\usepackage{mathrsfs}
\usepackage[noadjust]{cite}
\usepackage{microtype}
\usepackage{comment}
\usepackage{slashed}
\usepackage{mathtools}
\usepackage{centernot}
\usepackage{footnote}
\usepackage{enumerate}
\usepackage[shortlabels]{enumitem}
\usepackage{stackrel}
\usepackage{longtable}
\usepackage{cprotect}
\usepackage{xstring}
\usepackage{calc}
\usepackage{bbm}
\usepackage{dsfont}
\usepackage[colorlinks=true, pdfstartview=FitV, linkcolor=colorLink, citecolor=colorCite, urlcolor=colorLink, linktocpage=true]{hyperref}
\usepackage{upgreek}
\usepackage{parskip}
\usepackage{graphicx}
\usepackage{orcidlink}

\makeatletter
\renewcommand{\paragraph}{%
\@startsection{paragraph}{4}%
{\z@}{1.5ex \@plus 1.5ex \@minus .2ex}{-0.7em}%
{\normalfont\normalsize\bfseries}%
}
\makeatother

\makeatletter
\def\thm@space@setup{%
  \thm@preskip=\parskip \thm@postskip=0pt
}
\makeatother

\setlist[itemize]{leftmargin=5mm}

\DeclareSymbolFont{timesoperators}{T1}{ptm}{m}{n}
\SetSymbolFont{timesoperators}{bold}{T1}{ptm}{b}{n}
\DeclareMathAlphabet{\mathbb}{U}{jkpsyb}{m}{n}
\SetMathAlphabet{\mathbb}{bold}{U}{jkpsyb}{bx}{n}
\DeclarePairedDelimiter\abs\lvert\rvert

\allowdisplaybreaks
\definecolor{colorLink}{RGB}{0,100,162}
\definecolor{colorCite}{RGB}{8,124,100}

\def\N{\mathbb{N}}
\def\R{\mathbb{R}}

\def\E{\mathbb{E}}
\def\P{\mathbb{P}}

\def\CC{\mathcal{C}}

\def\CF{\mathcal{F}}

\def\CH{\mathcal{H}}

\def\CK{\mathcal{K}}
\def\CL{\mathcal{L}}
\def\CM{\mathcal{M}}
\def\CN{\mathcal{N}}
\def\CO{\mathcal{O}}
\def\CP{\mathcal{P}}

\def\CT{\mathcal{T}}

\def\CY{\mathcal{Y}}
\def\CZ{\mathcal{Z}}

\def\rmA{\mathrm{A}}

\def\rmD{\mathrm{D}}

\def\rmF{\mathrm{F}}

\def\rmL{\mathrm{L}}

\def\rmR{\mathrm{R}}

\def\rmW{\mathrm{W}}
\def\rmX{\mathrm{X}}
\def\rmY{\mathrm{Y}}
\def\rmZ{\mathrm{Z}}

\def\rmj{\mathrm{j}}

\def\frkK{\mathfrak{K}}

\def\one{\mathds{1}}

\let\eps\upvarepsilon
\let\alpha\upalpha
\let\beta\upbeta
\let\delta\updelta
\let\gamma\upgamma
\let\mu\upmu
\let\eta\upeta
\let\nu\upnu
\let\rho\uprho
\let\chi\upchi
\let\xi\upxi
\let\zeta\upzeta
\let\tau\uptau
\let\varphi\upvarphi
\let\lambda\uplambda
\let\theta\uptheta
\let\pi\uppi
\let\Upsilon\Upupsilon
\let\Theta\Uptheta
\let\Psi\Uppsi
\let\Xi\Upxi

\let\f\frac

\DeclareMathOperator{\loc}{loc}
\DeclareMathOperator{\gammac}{\gamma_{c}}

\newcommand{\eqdef}{\stackrel{\mbox{\tiny\rm def}}{=}}
\newcommand{\eqlaw}{\stackrel{\mbox{\tiny\rm law}}{=}}

\def\dash{\leavevmode\unskip\kern0.18em--\penalty\exhyphenpenalty\kern0.18em}
\def\slash{\leavevmode\unskip\kern0.15em/\penalty\exhyphenpenalty\kern0.15em}

\makeatletter
\renewcommand{\operator@font}{\mathgroup\symtimesoperators}
\makeatother

\makeatletter
\DeclareRobustCommand{\TitleEquation}[2]{\texorpdfstring{\StrLeft{\f@series}{1}[\@firstchar]$\if%
b\@firstchar\boldsymbol{#1}\else#1\fi$}{#2}}
\makeatother

\makeatletter
\newcommand{\pushright}[1]{\ifmeasuring@#1\else\omit\hfill$\displaystyle#1$\fi\ignorespaces}
\newcommand{\pushleft}[1]{\ifmeasuring@#1\else\omit$\displaystyle#1$\hfill\fi\ignorespaces}
\makeatother

\renewcommand{\bar}{\overline}
\renewcommand{\hat}{\widehat}

\newcommand{\ceps}{{(\eps)}}
\newcommand{\cepsn}{{(\eps_n)}}

\makeatletter
\newcommand{\oset}[3][0ex]{%
  \mathrel{\mathop{#3}\limits^{
    \vbox to#1{\kern-2\ex@
    \hbox{$\scriptstyle#2$}\vss}}}}
\makeatother

\newcommand{\frka}{\mathfrak{a}}

\newcommand{\frkh}{\mathfrak{h}}
\newcommand{\frkm}{\mathfrak{m}}

\usetikzlibrary{shapes.misc}
\usetikzlibrary{shapes.symbols}
\usetikzlibrary{shapes.geometric}
\usetikzlibrary{decorations}
\usetikzlibrary{decorations.markings}
\usetikzlibrary{calc}
\usetikzlibrary{external}
\usetikzlibrary{arrows}
\usetikzlibrary{patterns}

\theoremstyle{plain}
\newtheorem{theorem}{Theorem}[section]

\newtheorem{lemma}[theorem]{Lemma}
\newtheorem{proposition}[theorem]{Proposition}

\newtheorem{theoremA}{Theorem}[section]
\newtheorem{propositionA}[theoremA]{Proposition}

\theoremstyle{definition}
\newtheorem{definition}[theorem]{Definition}

\newtheorem*{acknowledgements}{Acknowledgements}

\newtheorem{remark}[theorem]{Remark}

\numberwithin{equation}{section}

\colorlet{darkblue}{blue!90!black}
\colorlet{darkgreen}{green!50!black}

\begin{document}

\title{Uniqueness of supercritical Gaussian multiplicative chaos}

\author{Federico Bertacco$^1$\orcidlink{0000-0002-6363-1294}, Martin Hairer$^{2}$\orcidlink{0000-0002-2141-6561}}

\institute{Imperial College London, Email: \href{mailto:f.bertacco20@imperial.ac.uk}{\color{black} \texttt{f.bertacco20@imperial.ac.uk}} \and EPFL \& Imperial College London, Email: \href{mailto:martin.hairer@epfl.ch}{\color{black}\texttt{martin.hairer@epfl.ch}}}

\maketitle

\begin{abstract}
We show that, for general convolution approximations to a large class of log-correlated Gaussian fields, the properly normalised supercritical Gaussian multiplicative chaos measures converge stably to a nontrivial limit. This limit depends on the choice of regularisation only through a multiplicative constant and can be characterised as an integrated atomic measure with a random intensity expressed in terms of the critical Gaussian multiplicative chaos.
\end{abstract}

\setcounter{tocdepth}{2}
\tableofcontents

\section{Introduction}\label{sec:intro}
For a domain $\rmD\subseteq \R^d$, with $d \geq 1$, a Gaussian Multiplicative Chaos (GMC) measure is (formally) a random measure of the form 
\begin{equation}
\label{e:GMCformal}
	\mu_{\gamma}(dx) \mathrel{\text{``$=$''}} e^{\gamma \rmX(x)} dx \;,
\end{equation}
where $\gamma >0$ is a real parameter, $dx$ denotes the Lebesgue measure on $\rmD$, and $\rmX$ is a log-correlated Gaussian field on $\rmD$. More precisely, $\rmX$ is a centred Gaussian field with covariance kernel $\CK : \rmD \times \rmD \to \R$ of the form
\begin{equation}
\label{eq:kerLGF}
 \CK(x, y) \eqdef \E \bigl[\rmX(x) \rmX(y)\bigr] = - \log\abs{x-y} + g(x,y)\;, \qquad \forall \, x, y \in \rmD \;,
\end{equation}
for some continuous function $g: \bar{\rmD} \times \bar{\rmD} \to \R$.
The main difficulty in defining rigorously the measure \eqref{e:GMCformal} is that, as one can see from \eqref{eq:kerLGF}, the covariance of $\rmX$ blows up along the diagonal, thus making $\rmX$ a rough centred Gaussian field that cannot be defined as a functional field, i.e., as a field indexed by points in the domain $\rmD$, but only makes sense as a distributional field, i.e., as a linear field indexed by test functions. 
To overcome this difficulty, thanks to the seminal work of Kahane \cite{Kahane}, there is by now a standard roadmap which involves a regularisation, renormalisation, and limiting procedure. 
More precisely, one first defines a collection of continuous pointwise defined centred Gaussian fields approximating $\rmX$. 
One then defines a sequence of properly renormalised random measures and finally passes to the limit. We refer to \cite{RV_Review, Shamov, Berestycki_Elementary, BerMulti, BP24} for more details on the topic.

It is now well known that the behaviour of the random measure \eqref{e:GMCformal} undergoes a phase transition at 
\begin{equation*}
\gammac \eqdef \sqrt{\smash[b]{2d}} \;.
\end{equation*}
More precisely, the regime $\gamma < \gammac$ is called \emph{subcritical}, the borderline case $\gamma = \gammac$ is called \emph{critical}, and the range $\gamma > \gammac$ is called \emph{supercritical}. These three regimes differ both in the normalisation required to achieve a nontrivial limiting measure and in the features of the resulting measure.

In this paper, we focus on the supercritical regime, where, to the best of our knowledge, the only existing mathematical reference in the continuum setting is \cite{Glassy} where the authors proved the convergence of a regularised and renormalised version of \eqref{e:GMCformal} to a nontrivial, purely atomic random measure which is \emph{not} measurable with respect to the underlying Gaussian field $\rmX$. More precisely, \cite{Glassy} establishes this convergence for a particular class of log-correlated Gaussian fields known as \emph{$\star$-scale invariant fields}, using a specific approximation called the \emph{martingale approximation}. The primary objective of this paper is to extend this convergence result to a large class of log-correlated Gaussian fields and their convolution approximations.

\subsection{Definitions and assumptions}
Before stating our main results, we introduce some definitions. We begin by recalling the definition of the convolution approximation of a general log-correlated Gaussian field $\rmX$ defined on a domain $\rmD \subseteq \R^d$ with a covariance kernel of the form \eqref{eq:kerLGF}. 
Specifically, we consider a mollifier $\smash{\rho: \R^d \to \R}$ and, denoting by $\smash{\hat \rho}$ its Fourier transform, we assume that it satisfies the following conditions:
\begin{enumerate}[start=1,label={{{(A\arabic*})}}]
\item \label{hp_A1} $\rho \in \CC_c^{\infty}(\R^d)$ and has unit mass.
\item \label{hp_A2} For every nonzero multi-index $\rmj = (\rmj_1, \ldots, \rmj_d) \in \N_0^d$ with $\abs{\rmj} \leq d-1$, it holds that $\partial^{\rmj}\hat \rho(0) = 0$, and $\abs{\hat \rho}$ attains a local maximum at $0$.
\end{enumerate}

\begin{remark}
The conditions in \ref{hp_A2} on $\hat{\rho}$ are of a purely technical nature and are used in the proof of Proposition~\ref{pr:decoCovStar}. 
\end{remark}

\begin{remark}
We note that although condition \ref{hp_A2} is crucial for our approach, it is unclear to us whether it is truly necessary for the conclusion of Theorem~\ref{th:convergence} to hold. We also observe that in the case $d = 2$, if $\rho \in \CC^{\infty}_c(\R^d)$ is even, non-negative, and has unit mass, then assumption \ref{hp_A2} is satisfied.
\end{remark}

\begin{remark}
Since the Taylor expansion of $\hat \rho$ at the origin is determined by the moments of $\rho$, it is easy to
use the same argument as in \cite[Exercise~13.8]{Book} to construct even  functions $\rho\in \CC_c^{\infty}(\R^d)$ such that $\hat{\rho}(\omega) = 1 - \abs{\omega}^{2m} + \CO(\abs{\omega}^{2m+1})$, thus satisfying both of the assumptions above.
\end{remark}

For $\eps > 0$, we set $\smash{\rho_{\eps}(\cdot) \eqdef \eps^{-d} \rho(\eps^{-1} \cdot)}$. We define the convolution approximation $\rmX_{\ceps}$ of $\rmX$ as follows:
\begin{equation}
\label{eq:conv_version_field}
\rmX_{\ceps} \eqdef \rho_{\eps} * \rmX \;.
\end{equation}
One can easily verify that the covariance kernel of the field $\rmX_{\ceps}$ is given by
\begin{equation*}
\CK^{\rho}_{\ceps}(x, y)  \eqdef \E \bigl[\rmX_{\ceps}(x) \rmX_{\ceps}(y)\bigr] = \bigl((\rho_{\eps} \otimes \rho_{\eps}) * \CK \bigr) (x, y)\;, \qquad \forall \, x, y \in \rmD \;.
\end{equation*}

We now introduce $\star$-scale invariant fields and their martingale approximations. 
The key ingredient in constructing a $\star$-scale invariant field is the so-called \emph{seed covariance function} $\frkK:\R^d \to \R$. We assume that $\frkK$ and its Fourier transform $\hat{\frkK}$ satisfy the following properties:
\begin{enumerate}[start=1,label={{{(K\arabic*})}}]
\item \label{hp_K1} $\frkK$ is positive definite, radial, and $\frkK(0) = 1$.
\item \label{hp_K2} $\frkK \in \CC^{\infty}(\R^d)$ and $\abs{\frkK(x)} \leq C(1+\abs{x})^{-a}$ for some $C$, $a > 0$ and all $x \in \R^d$.
\item \label{hp_K3} $\hat{\frkK}$ is supported in $B(0, 1)$\footnote{Note that for this to hold, the kernel $\frkK$ must \emph{not} be compactly supported.} and $\bar{\frka} \abs{\omega}^d \leq \int_{\abs{\xi} \leq \abs{\omega}} \hat{\frkK}(\xi) d\xi \leq \underline{\frka} \abs{\omega}^d $ for all $\abs{\omega} < 1$ and some $\underline{\frka}$, $\bar{\frka} > 0$.  
\end{enumerate}

\begin{remark}
We emphasise that the assumption $\frkK \in \CC^{\infty}(\R^d)$ in \ref{hp_K2} follows directly from the requirement that $\hat{\frkK}$ be compactly supported in \ref{hp_K3}. The reason for this redundancy is that, in some parts of the paper, we consider $\frkK$ satisfying only conditions \ref{hp_K1} \dash \ref{hp_K2}.
\end{remark}

\begin{remark}
\label{rem:SeedWorksFourier}
An example of a seed covariance function $\frkK: \R^d \to \R$ satisfying assumptions \ref{hp_K1} \dash \ref{hp_K3} is given by the inverse Fourier transform of the (normalised) indicator function of the unit ball.
More precisely, let $\smash{\hat{\frkK}: \R^d \to \R}$ and $\smash{\frkK: \R^d \to \R}$ be defined as follows:
\begin{equation*}
	\hat{\frkK}(\omega) \eqdef \frac{1}{\abs{B(0, 1)}}\one_{\{\abs{\omega} \leq 1\}}\;, \qquad \frkK(x) \eqdef \frac{1}{\abs{B(0, 1)}}\int_{B(0, 1)} e^{i \omega \cdot x} d \omega \;.
\end{equation*}
Then, $\frkK$ is positive definite since $\hat{\frkK}$ is non-negative, it is radial since the inverse Fourier transform of a radial function, and $\frkK(0) = 1$. Additionally, $\frkK \in \CC^{\infty}(\R^d)$ since $\hat{\frkK}$ has compact support, and one can easily verify that $\abs{\frkK(x)} \lesssim (1+\abs{x})^{-(d+1)/2}$ for all $x \in \R^d$. The conditions in \ref{hp_K3} are trivially satisfied. 
\end{remark}

\begin{remark}
We caution the reader against interpreting the seemingly restrictive assumptions \ref{hp_K1} \dash \ref{hp_K3} on the seed covariance kernel as a limitation of our approach. As will become evident in the subsequent sections, it suffices to construct the supercritical GMC under the convolution approximation for a single $\star$-scale invariant field. This construction then allows us to generalise the result to a large class of log-correlated Gaussian fields.  In particular, one may assume throughout the remainder of the paper that $\rmX^{\star}$ is the $\star$-scale invariant field with seed covariance function as specified in Remark~\ref{rem:SeedWorksFourier}.
\end{remark}

\begin{definition}
\label{def:fields}
Let $\frkK : \R^d \to \R$ be a function satisfying assumptions \ref{hp_K1} \dash \ref{hp_K2}\footnote{Note that assumption \ref{hp_K3} is omitted here, as it is not required for the definition of $\star$-scale invariant fields.}, and let $\bar\frkK:\R^d \to \R$ denote the (unique) positive definite function whose convolution with itself equals $\frkK$. For $\xi$ a space-time white noise on $\R^d \times (0, \infty)$, we define the \emph{$\star$-scale invariant field with seed covariance $\frkK$} by
\begin{equation}
\label{eq:field}
\rmX^{\star}(\cdot) \eqdef \int_{\R^d} \int_0^{\infty} \bar\frkK\bigl(e^{r}(y - \cdot)\bigr) e^{\f{dr}{2}} \xi(dy, dr) \;.
\end{equation}
Furthermore, for $t \geq 0$, we let $\rmX^{\star}_{t}$ be the field on $\R^d$ given by
\begin{equation}
\label{eq:fieldTTT}
\rmX^{\star}_{t}(\cdot) \eqdef \int_{\R^d} \int_{0}^{t} \bar\frkK\bigl(e^{r}(y - \cdot)\bigr) e^{\f{dr}{2}} \xi(dy, dr) \;.
\end{equation}
\end{definition}
For all $x$, $y \in \R^d$ and $s$, $t \geq 0$, it holds that 
\begin{equation}
\label{eq:covs}
\E\bigl[\rmX^{\star}(x) \rmX^{\star}(y)\bigr] = \int_0^{\infty} \frkK\bigl(e^{r}(x-y)\bigr) dr \;, \qquad \E\bigl[\rmX^{\star}_s(x) \rmX^{\star}_t(y)\bigr] = \int_0^{s \wedge t} \frkK\bigl(e^{r}(x-y)\bigr) dr \;.
\end{equation}
The collection of fields $(\rmX^{\star}_t)_{t \geq 0}$ is called the \emph{martingale approximation} of $\rmX^{\star}$. Indeed, 
by construction, $\smash{(\rmX^{\star}_t)_{t \geq 0}}$ is a martingale for the filtration $\smash{(\CF_t)_{t \geq 0}}$ given by 
\begin{equation}
\label{eq:filtrationStar}
\CF_t \eqdef \sigma\bigl(\rmX^{\star}_{s} \; : \; s \in [0, t)\bigr)	\;.
\end{equation}
Moreover, as $t \to \infty$ the field $\rmX^{\star}_t$ converges almost surely to $\rmX^{\star}$ in the Sobolev space $\CH^{-\kappa}(\R^d)$, for any $\kappa > 0$. We refer to \cite[Proposition~4.1~(iv)]{Junnila_deco} for a proof of this fact.

\begin{remark}
In what follows, for a $\star$-scale invariant field $\rmX^{\star}$ and $t \geq 0$, we always write $\rmX^{\star}_t$ for the martingale approximation of $\rmX^{\star}$ at level $t$. For the convolution approximation, we always indicate the ``smoothing parameter'' in brackets, i.e., we write $\rmX^{\star}_{\ceps}$ to denote the convolution regularisation of $\rmX^{\star}$ at level $\eps > 0$.
\end{remark}

\subsection{Main results}
\label{sub:main_results}
Let $\rmX$ be a log-correlated Gaussian field defined on a bounded domain $\rmD \subseteq \R^d$ with a covariance kernel of the form \eqref{eq:kerLGF}, and let $(\rmX_{\ceps})_{\eps > 0}$ denote its convolution approximation as defined in \eqref{eq:conv_version_field}. For $\gamma > \sqrt{\smash[b]{2d}}$ and $\eps > 0$, we define the random measure $\mu_{\gamma, \ceps}$ on $\rmD$ by letting 
\begin{equation}
\label{e:seq_super_main}
\mu_{\gamma, \ceps}(dx) \eqdef \abs{\log \eps}^{\f{3\gamma}{2\sqrt{\smash[b]{2d}}}} \eps^{-(\gamma/\sqrt 2-\sqrt{d})^2} e^{\gamma \rmX_{\ceps}(x) - \frac{\gamma^2}{2} \E[{\rmX_{\ceps}(x)}^2]} dx \;.
\end{equation}

Before stating our main result, we introduce some additional notation.
\begin{definition}
\label{def:PPP}
For $\gamma > \sqrt{\smash[b]{2d}}$ and a non-negative, locally finite Borel measure $\nu$ on $\R^d$, we let $\eta_{\gamma}[\nu]$ be the Poisson point measure on $\R^d \times (0, \infty)$ with intensity measure given by $\smash{\nu(dx) \otimes z^{-(1+\gammac/\gamma)} dz}$. We also define the integrated atomic random measure with parameter $\gamma$ and spatial intensity $\nu$ as the random purely atomic measure $\CP_{\gamma}[\nu]$ on $\R^d$ given by
\begin{equation*}
\CP_{\gamma}[\nu](dx) \eqdef \int_{0}^{\infty} z \, \eta_{\gamma}[\nu](dx, dz) \;. 
\end{equation*}
\end{definition}

In what follows, $\mu_{\gammac}$ denotes the critical GMC associated with $\rmX$, obtained through the derivative normalisation\footnote{It is well-known that this measure does not depend on the particular choice of the approximation scheme used to define it.} \cite{Critical_der, Junnila_deco, Powell_Critical}. Furthermore, we also introduce the measure $\bar\mu_{\gammac}$ by setting
\begin{equation}
\label{eq:tiltedCritical}
\bar\mu_{\gammac}(dx) \eqdef e^{(d - \sqrt{\smash[b]{d/2}}\gamma) g(x, x)} \mu_{\gammac}(dx)	
\end{equation}
where $g: \bar\rmD \times \bar\rmD \to \R$ is the function appearing in \eqref{eq:kerLGF}. 

Throughout this paper, we denote by $\CH^s(\R^d)$ the standard $L^2$-based Sobolev space with smoothness index $s \in \R$. Furthermore, given a domain $\rmD \subseteq \R^d$, we define the local Sobolev space $\CH^s_{\loc}(\rmD)$ as the space of distributions whose pairings with all test functions in $\CC^{\infty}_c(\rmD)$ belong to $\CH^s(\R^d)$. 

Referring to Definition~\ref{def:stabelConv} for the notion of stable convergence of a sequence of random measures, we are now ready to state the main result of this paper.
\begin{theoremA}
\label{th:convergence}
Let $\rmX$ be a log-correlated Gaussian field on a bounded domain $\smash{\rmD \subseteq \R^d}$ with covariance kernel of the form \eqref{eq:kerLGF}, where $\smash{g \in \CH^s_{\loc}(\rmD \times \rmD)}$ for some $s > d$. Let $\smash{\rho: \R^d \to \R}$ be a mollifier satisfying assumptions \ref{hp_A1} \dash \ref{hp_A2}, and let $\smash{(\rmX_{\ceps})_{\eps > 0}}$ be the convolution approximation of $\rmX$ as defined in \eqref{eq:conv_version_field}. For $\smash{\gamma > \sqrt{\smash[b]{2d}}}$, consider the sequence of random measures $\smash{(\mu_{\gamma, \ceps})_{\eps > 0}}$ on $\rmD$ defined in \eqref{e:seq_super_main}.
Then, there exists a constant $a_{\gamma, \rho} > 0$, depending on $\gamma$ and the mollifier $\rho$, such that $\mu_{\gamma, \ceps}$ converges $\sigma(\rmX)$-stably to $\smash{a_{\gamma, \rho} \CP_{\gamma}[\bar{\mu}_{\gammac}]}$ along any sequence $(\eps_n)_{n \in \N}$ converging to $0$, where $\smash{\bar{\mu}_{\gammac}}$ is the measure defined in \eqref{eq:tiltedCritical}.
\end{theoremA}

\begin{remark}
The convergence in Theorem~\ref{th:convergence} indicates that the limiting measure can be formally decomposed into two components: the location of the point masses, which is determined by an instance of the field $\rmX$ through the associated critical GMC, and an additional source of randomness that is independent of $\rmX$ and controls the weights of the point masses. Importantly, this implies that the sequence $\mu_{\gamma, \ceps}$ does not converge in probability.
\end{remark}

\begin{remark}
Stable convergence has been previously used in the theory of GMC. For instance, Lacoin employed it in \cite{LacoinIII} for the study of complex GMC.
\end{remark}

The proof of Theorem~\ref{th:convergence} relies on \cite[Theorem~C]{BHCluster} and on the following new decomposition result.
\begin{propositionA}
\label{pr:decoCovStar}
Let $\rho$ be a mollifier satisfying assumptions \ref{hp_A1} \dash \ref{hp_A2}, and let $\rmX^{\star}$ be a $\star$-scale invariant field whose seed covariance function $\frkK$ satisfies assumptions \ref{hp_K1} \dash \ref{hp_K3}. Then, there exist a constant $\frka = \frka(\rho, \frkK) \in (0,1)$ and a smooth, stationary, centred Gaussian field $\rmW$, with decaying correlations\footnote{By decaying correlations, we mean that $\lim_{|x-y| \to \infty} \E[\rmW(x)\rmW(y)] = 0$.}, and independent of $\CF_{\infty}$ (recall \eqref{eq:filtrationStar}), such that for any fixed $\eps \in (0, \frka)$, defining $\smash{t_{\eps} \eqdef \log(\frka \eps^{-1})}$, the following decomposition holds 
\begin{equation}
\label{e:decoCovStar}
\rmX^{\star}_{\ceps} \eqlaw \rmX^{\star}_{t_{\eps}} + \rmW_{t_{\eps}} + \rmZ_{t_{\eps}} \;,
\end{equation}
where the fields on the right-hand side of \eqref{e:decoCovStar} are all mutually independent, and for all $\eps \in (0, \frka)$, it holds that:
\begin{enumerate}
	\item The field $\rmW_{t_{\eps}}$ is such that $\rmW_{t_{\eps}}(\cdot) = \rmW(e^{t_{\eps}} \cdot)$.
	\item The field $\rmZ_{t_{\eps}}$ is a smooth, stationary, centred Gaussian field such that $\lim_{\eps \to 0} \E[\rmZ_{t_{\eps}}(\cdot)^2] = 0$.
\end{enumerate}
\end{propositionA}  

\begin{remark}
It is important to emphasise that, in Proposition~\ref{pr:decoCovStar}, we are \emph{not} claiming that, for any fixed $x \in \R^d$, the law of the process $\smash{(0, \frka) \ni \eps \mapsto \rmX^{\star}_{(\eps)}(x)}$ coincides with the law of the process $\smash{(0, \frka) \ni \eps \mapsto \rmX^{\star}_{t_{\eps}}(x) + \rmW_{t_{\eps}}(x) + \rmZ_{t_{\eps}}(x)}$. Rather, we are claiming that for any fixed $\eps > 0$, the field $\smash{(\rmX^{\star}_{(\eps)}(x))_{x \in \R^d}}$ has the same law as the field $\smash{(\rmX^{\star}_{t_{\eps}}(x) + \rmW_{t_{\eps}}(x) + \rmZ_{t_{\eps}}(x))_{x \in \R^d}}$. This is sufficient for our purpose.
\end{remark}

\begin{remark}
Since the field $\rmZ_{t_\eps}$ vanishes in the limit $\eps \to 0$, the decomposition \eqref{e:decoCovStar} roughly indicates that the convolution approximation of $\rmX^{\star}$ equals its martingale approximation plus an independent field that asymptotically behaves like white noise with finite variance.
\end{remark}

\begin{remark}
Although this article focuses on the convolution approximation, it is worth emphasising that the conclusion of Theorem~\ref{th:convergence} should also hold for any approximation scheme that admits a decomposition similar to \eqref{e:decoCovStar}.
\end{remark}

To the best of our knowledge, the decomposition established in Proposition~\ref{pr:decoCovStar} is new, and we hope that it will prove useful in other contexts as well. For example, it would be interesting to explore whether similar techniques to those employed in the current paper could be used to extend the main result of \cite{Madaule_Max} to the convolution approximation of any log-correlated Gaussian field. 

\subsection{Outline}
The remainder of this paper is organised as follows.  
In Section~\ref{sec:Setup}, we collect preliminary results that will be used throughout the paper.  
In Section~\ref{sec:proofs}, we prove the main results of this work. Specifically, we begin with the proof of Theorem~\ref{th:convergence}, followed by the proof of Proposition~\ref{pr:decoCovStar}.  
In Section~\ref{sec:RemovingCompact}, we show that the convergence result in \cite[Theorem~C]{BHCluster} can be generalised to accommodate fields with infinite-range correlations.  
Finally, in Appendix~\ref{sec:moments}, we present some results concerning moments and the multifractal spectrum of supercritical GMC measures.

\begin{acknowledgements}
{\small We are grateful to Christophe Garban for interesting discussions during the early stages of this project.
Both authors were supported by the Royal Society through MH's Research Professorship RP\textbackslash R1\textbackslash 191065.}
\end{acknowledgements}

\section{Background and preliminaries}
\label{sec:Setup}
In this section, we gather some background material and preliminary results. More precisely, after introducing some recurring basic notations, we discuss the concept of stable convergence in Section~\ref{sub:topPrel} and state some of its key properties. Finally, in Section~\ref{sub:decoFund}, we present a fundamental decomposition result for log-correlated Gaussian fields established in \cite{Junnila_deco}.

\subsection{Basic notation}
We adopt the convention to let $\N = \{1,2, \ldots\}$ and $\N_0 = \{0, 1, 2, \ldots\}$. We let $\R^{+} = (0, \infty)$ and $\R^{+}_0 = [0, \infty)$.
For a domain $\rmD \subseteq \R^d$, we write $\CC(\rmD)$ (resp.\ $\CC_c(\rmD)$) for the space of continuous (resp.\ continuous with compact support) functions from $\rmD$ to $\R$. We write $\CC^{\infty}(\rmD)$ (resp.\ $\CC^{\infty}_c(\rmD)$) for the space of smooth (resp.\ smooth with compact support) functions from $\rmD$ to $\R$. We write $\CC^{+}_c(\rmD)$ for the space of positive continuous functions from $\rmD$ to $\R$ with compact support. We let $\CM^{+}(\rmD)$ be the space of non-negative, locally finite measures on $\rmD$. Given a measure $\nu$ and a function $f$, we write $\nu(f)$ to denote the integral of $f$ against $\nu$.

\subsection{Topological preliminaries}
\label{sub:topPrel}
We now recall some facts about stable convergence of random measures. This type of convergence, differently from the convergence in distribution, is a convergence of the sequence of random variables itself rather than of the sequence of their distributions. We refer to the monographs \cite{JacodLimit, StableBook} and references therein for more details on stable convergence in a more general setting.

For a domain $\rmD \subseteq \R^d$, we equip the space $\CM^{+}(\rmD)$ of non-negative, locally finite measures on $\rmD$ with the topology of vague convergence. We equip the space of probability measures on $\CM^{+}(\rmD)$ with the topology of weak convergence.
For a sequence $(\nu_n)_{n \in \N}$ of $\CM^{+}(\rmD)$-valued random variables, we write $\nu_{n} \Rightarrow \nu$ to indicate that $\nu_{n}$ converges \emph{vaguely in distribution} to $\nu$ in $\CM^{+}(\rmD)$ as $n \to \infty$.

We consider a collection $(\nu_{n})_{n \in \N}$ of $\CM^{+}(\rmD)$-valued random variables defined on
a common probability space $(\Omega,\P)$ and a $\CM^{+}(\rmD)$-valued random variables $\nu$ defined
on a possibly larger probability space. We also fix a $\sigma$-algebra $\Sigma$ over $\Omega$.

\begin{definition}
\label{def:stabelConv}
We say that $\nu_{n}$ converges $\Sigma$-stably to $\nu \in \CM^{+}(\rmD)$ as $n \to \infty$, if $(Z, \nu_{n}) \Rightarrow (Z, \nu)$ as $n \to \infty$ for all $\Sigma$-measurable random variables $Z$. 
\end{definition}

\begin{remark}
If $\Sigma$ is the trivial $\sigma$-algebra, then this coincides with convergence in law. Conversely, if $\Sigma$ is the full $\sigma$-algebra of the probability space $\Omega$ and the limiting random variable is defined on $(\Omega, \P)$, then this corresponds to convergence in probability. Therefore, one can view stable convergence as a mode of convergence that lies between convergence in distribution and convergence in probability.
\end{remark}

Given a random variable $\rmY$ defined on the same probability space as above and taking values in a Polish space $\CY$, we have the following result that characterises $\sigma(\rmY)$-stable convergence.  
\begin{lemma}
\label{lm:stableRV}
Consider the same setting described above. Then $\nu_{n}$ converges $\sigma(\rmY)$-stably to $\nu$ if and only if $\smash{(\rmY, \nu_{n}) \Rightarrow (\rmY, \nu)}$ as $n \to \infty$.
\end{lemma}
\begin{proof}
See for instance \cite[Exercise~3.11]{StableBook}.
\end{proof}

We will also need the following lemma.
\begin{lemma}
\label{lm:doubleStable}
In the setting of Lemma~\ref{lm:stableRV}, let $\rmZ$ be a random variable taking values in a Polish space $\CZ$, and suppose that each $\nu_n$ is conditionally independent of $\rmZ$ given $\rmY$. If $\nu_{n}$ converges $\sigma(\rmY)$-stably to $\nu$ as $n \to \infty$, then it also converges $\sigma(\rmY,\rmZ)$-stably to $\nu$ as $n \to \infty$, and $\nu$ is conditionally independent of $\rmZ$ given $\rmY$. 
\end{lemma}
\begin{proof}
Let $\CL(\rmZ | \rmY)$ denote the conditional law of $\rmZ$ given $\rmY$. This is a $\sigma(\rmY)$-measurable random variable taking values in the (Polish) space of probability measures $\CP(\CZ)$ endowed with the weak convergence topology.
 Since $\nu_{n}$ converges $\sigma(\rmY)$-stably to $\nu$ as $n \to \infty$, it follows from \cite[Theorem~3.17 (vii)]{StableBook} that the pair $(\nu_n, \CL(\rmZ | \rmY))$ also converges $\sigma(\rmY)$-stably to $(\nu, \CL(\rmZ | \rmY)) $ as $n \to \infty$. 
 In particular, it follows from \cite[Exercise~3.11]{StableBook} that we have the following joint convergence in law as $n \to \infty$, 
\begin{equation}
\label{eq:StableStr1}
\bigl(\nu_n, \rmY, \CL(\rmZ | \rmY)\bigr) \Rightarrow \bigl(\nu, \rmY, \CL(\rmZ | \rmY)\bigr) \;.
\end{equation} 
To verify that $\nu_{n}$ converges $\sigma(\rmY,\rmZ)$-stably to $\nu$ as $n \to \infty$, it suffices, by \cite[Exercise~3.12]{StableBook}, to check that for any bounded continuous function $\rmF : \CM^{+}(\rmD) \times \CY \times \CZ \to \R$, it holds that 
\begin{equation*}
\lim_{n \to \infty} \E\bigl[\rmF(\nu_n, \rmY, \rmZ)\bigr] = \E\bigl[\rmF(\nu, \rmY, \rmZ)\bigr]\;.
\end{equation*}
To this end, we define the function $\bar \rmF : \CM^{+}(\rmD) \times \CY \times \CP(\CZ) \to \R$ by
\begin{equation*}
\bar{\rmF}(\rho, y, \CL) \eqdef \int_{\CZ} \rmF(\rho, y, z) \CL(dz) \;.
\end{equation*}
Thanks to the conditional independence of $\nu_n$ and $\rmZ$ given $\rmY$, and by Fubini's theorem, we have for any $n \in \N$ that
\begin{equation*}
\E\bigl[\rmF(\nu_n, \rmY, \rmZ)\bigr]  = \E\Biggl[\int_{\CZ} \rmF(\nu_n, \rmY, z)\CL(\rmZ | \rmY)(d z) \Biggr]  = \E\bigl[\bar{\rmF}(\nu_n, \rmY, \CL(\rmZ | \rmY))\bigr] \;.
\end{equation*} 
By \eqref{eq:StableStr1}, we have that
\begin{equation*}
 \lim_{n \to \infty} \E\bigl[\bar{\rmF}(\nu_n, \rmY, \CL(\rmZ | \rmY))\bigr] = \E\bigl[\bar{\rmF}(\nu, \rmY, \CL(\rmZ | \rmY))\bigr] \;,
\end{equation*} 
so that the desired result follows immediately.
\end{proof}

In the next lemma, we highlight an important feature of stable convergence. Specifically, the following result can be viewed as a generalisation of the classical Slutsky's theorem within the context of stable convergence.
\begin{lemma}
\label{lm:SlutskyStable}
Let $\rmA \subset \R^d$ be a compact set. Consider a sequence $\smash{(\nu_{n})_{n \in \N}}$ of $\smash{\CM^{+}(\rmA)}$-valued random variables and a sequence $\smash{(\frkh_{n})_{n \in \N}}$ of random variables taking values in $\CC(\rmA)$ (equipped with the topology of local uniform convergence). For a random variable $\rmY$, assume that $\smash{\nu_n}$ converges $\sigma(\rmY)$-stably to $\nu$, and $\frkh_n$ converges in probability to a $\sigma(\rmY)$-measurable $\CC(\rmA)$-valued random variable $\frkh$. Then it holds that $\smash{\frkh_n \nu_n}$ converges $\sigma(\rmY)$-stably to $\smash{\frkh \nu}$ as $n \to \infty$. 
\end{lemma}
\begin{proof}
The proof of this result follows from \cite[Theorem~3.18~(b)]{StableBook} and the continuous mapping theorem \cite[Theorem~3.18~(c)]{StableBook}. 
\end{proof}

We record here the following useful fact.
\begin{remark}
\label{rem:comp_Laplace}
For $\gamma > \sqrt{\smash[b]{2d}}$ and a Radon measure $\nu$ on $\rmD$, let $\CP_{\gamma}[\nu]$ be the integrated atomic random measure with parameter $\gamma$ and spatial intensity $\nu$ as specified in Defnition~\ref{def:PPP}.
Then, for every $\varphi \in \CC_c^{+}(\rmD)$, it holds that
\begin{align}
\E\bigl[\exp\bigl(-\CP_{\gamma}[\nu](\varphi)\bigr)\bigr]
& = \E\Biggl[\exp\Biggl(- \int_{\rmD \times \R^{+}}  \varphi(x) z \, \eta_{\gamma}[\nu](dx, dz) \Biggr)\Biggr]\nonumber \\
& = \exp\Biggl(- \int_{\rmD \times \R^{+}} \f{1-e^{-\varphi(x) z}}{z^{1+\sqrt{\smash[b]{2d}}/\gamma}} \nu(dx) dz \Biggr) \nonumber \\
& = \exp\Biggl(- \beta(d, \gamma) \int_{\rmD} \varphi(x)^{\f{\sqrt{\smash[b]{2d}}}\gamma}  \nu(dx) \Biggr) \;, \label{eq:comp_Laplace} 
\end{align}
where $\smash{\beta(d, \gamma) \eqdef \Gamma(1-\sqrt{\smash[b]{2d}}/\gamma)/(\sqrt{\smash[b]{2d}}/\gamma)}$.
\end{remark}

We conclude this section with the following result which will be used in the proof of Theorem~\ref{th:convergence} below.
\begin{lemma}
\label{lem:mult}
For $\gamma > \sqrt{\smash[b]{2d}}$ and $\nu \in \CM^{+}(\rmD)$, let $\CP_{\gamma}[\nu]$ be the integrated atomic random measure with parameter $\gamma$ and spatial intensity $\nu$ as defined in Definition~\ref{def:PPP}. Let $f \in \CC(\rmD)$ be a non-negative, continuous function. Then, one has that
\begin{equation*}
f \CP_{\gamma}[\nu] \eqlaw \CP_{\gamma}\bigl[f^{\sqrt{\smash[b]{2d}}/\gamma} \nu\bigr]\;.
\end{equation*}
Furthermore, the same conclusion holds if $f \in \CC(\rmD)$ is a non-negative, random, continuous function independent of $\eta_{\gamma}[\nu]$. 
\end{lemma}
\begin{proof}
The proof follows by an immediate computation (based on \eqref{eq:comp_Laplace}) of the Laplace functional of the random measure $f \CP_{\gamma}[\nu]$.
\end{proof}

\subsection{A decomposition result for log-correlated Gaussian fields}
\label{sub:decoFund}
In this short section, we recall a key decomposition result for log-correlated Gaussian fields proved in \cite[Theorem~A]{Junnila_deco}. Roughly speaking, this general theorem states that given two such fields, under some suitable mild regularity assumptions on the covariance kernels, there exists a coupling between them in such a way that their difference is given by a H\"older continuous centred Gaussian field. 

In our specific setting, this result can be stated as follows.
\begin{proposition}
\label{pr:decoGeneral}
For a domain $\rmD\subseteq \R^d$, let $\rmX$ be a log-correlated Gaussian field with covariance kernel of the form \eqref{eq:kerLGF} with $\smash{g \in \CH^s_{\loc}(\rmD\times\rmD)}$, for some $s>d$. Let $\rmX^{\star}$ be a $\star$-scale invariant field whose seed covariance satisfies \ref{hp_K1}\dash\ref{hp_K2}. Then, for any bounded domain $\rmD'$ with closure satisfying $\overline{\rmD'} \subset \rmD$, we can construct copies of the fields $\rmX$ and $\rmX^{\star}$ on a common probability space in such a way that the following decomposition holds on $\rmD'$
\begin{equation}
\label{e:decoFieldsGeneral}
\rmX = \rmX^{\star} + \rmL  \;,
\end{equation}
where $\rmL$ is a centred Gaussian field which is almost surely H\"older continuous on $\rmD'$, and $\rmX^{\star}$ and $\rmL$ are jointly Gaussian.  
\end{proposition}
\begin{proof}
This result is an immediate consequence of \cite[Theorem~A, Lemma~3.4]{Junnila_deco}. In particular, the fact that the covariance of $\rmX^{\star}$ satisfies the assumptions of \cite[Theorem~A]{Junnila_deco} follows from \cite[Proposition~4.1~(vi)]{Junnila_deco}.
\end{proof}

\section{Proof of main results}
\label{sec:proofs}
In this section, we establish our main results. In Section~\ref{sub:proofUniq}, we first prove the uniqueness result stated in Theorem~\ref{th:convergence}, relying on Proposition~\ref{pr:decoCovStar}, which will be proved later in Section~\ref{sub:conv}.

\subsection{Proof of Theorem~\ref{th:convergence}}
\label{sub:proofUniq}
We fix a log-correlated Gaussian field $\rmX$ on a bounded domain $\smash{\rmD \subseteq \R^d}$ with covariance kernel of the form \eqref{eq:kerLGF} with $\smash{g \in \CH^s_{\loc}(\rmD\times\rmD)}$, for some $s>d$. Furthermore, let $\rho: \R^d \to \R$ be a mollifier satisfying assumptions \ref{hp_A1} \dash\ref{hp_A2} and let $\smash{(\rmX_{\ceps})_{\eps > 0}}$ be the convolution approximation of $\rmX$ constructed using $\rho$, as defined in \eqref{eq:conv_version_field}. We also fix a $\star$-scale invariant field $\rmX^{\star}$ with seed covariance function $\frkK$ satisfying \ref{hp_K1} \dash \ref{hp_K3}, and we denote by $(\rmX_t^{\star})_{t \geq 0}$ its martingale approximation. We let $\smash{(\rmX^{\star}_{\ceps})_{\eps > 0}}$ be the convolution approximation of $\smash{\rmX^{\star}}$ constructed using $\rho$. For $\gamma > \sqrt{\smash[b]{2d}}$, we consider the collection of measures $(\mu^{\star}_{\gamma, \ceps})_{\eps >0}$ on $\R^d$ defined as follows
\begin{equation}
\label{e:seq_super_main_star} 
\mu^{\star}_{\gamma, \ceps}(dx) \eqdef \abs{\log \eps}^{\f{3\gamma}{2\sqrt{\smash[b]{2d}}}} \eps^{-(\gamma/\sqrt 2-\sqrt{\smash[b]{d}})^2} e^{\gamma \rmX^{\star}_{\ceps}(x) - \frac{\gamma^2}{2} \E[{\rmX^{\star}_{\ceps}(x)}^2]} dx \;.
\end{equation}

Thanks to Proposition~\ref{pr:decoGeneral}, we can (and will) construct copies of $\rmX$ and $\rmX^{\star}$ on a common probability space in such a way that the following decomposition holds:
\begin{equation}
\label{e:decoFieldsGeneralApp}
\rmX = \rmX^{\star} + \rmL  \;,
\end{equation}
where $\rmL$ is a centred Gaussian field which is almost surely H\"older continuous. In particular, an immediate consequence of \eqref{e:decoFieldsGeneralApp} is that, for all $\eps > 0$, it holds that 
\begin{equation}
\label{eq:decoFinalApp}
\rmX_{\ceps} = \rmX^{\star}_{\ceps} + \rmL_{\ceps} \;,
\end{equation}
where $\rmL_{\ceps} \eqdef \rho_{\eps} * \rmL$. For $\eps > 0$, we introduce the random function $\frkh_{\gamma, \eps} : \rmD \to \R$ by letting
\begin{equation}
\label{eq:smoothBit}
\frkh_{\gamma, \ceps}(x) \eqdef e^{\gamma \rmL_{\ceps}(x) - \frac{\gamma^2}{2} \E[\rmX_{\ceps}(x)^2 - {\rmX^{\star}_{\ceps}(x)}^2]} \;.
\end{equation}
Thanks to \eqref{eq:decoFinalApp}, we observe that we can rewrite the measure $\mu_{\gamma, \ceps}(dx)$ defined in \eqref{e:seq_super_main} as follows
\begin{equation}
\label{eq:compConvMol}
\mu_{\gamma, \ceps}(dx) = \frkh_{\gamma, \ceps} \mu^{\star}_{\gamma, \ceps}(dx) \;,
\end{equation}
where the measure $\mu^{\star}_{\gamma, \ceps}(dx)$ is defined as in \eqref{e:seq_super_main_star}. Hence, to study the convergence of $\mu_{\gamma, \ceps}$, it suffices to separately analyse the convergence of  $\frkh_{\gamma, \ceps}$ and that of $\mu^{\star}_{\gamma, \ceps}$.  In the following lemma, we focus on the latter. 

\begin{lemma}
\label{lm:stableConvStar}
Consider the setting described above. For $\gamma > \sqrt{\smash[b]{2d}}$, there exists a finite constant $\smash{a^{\star}_{\gamma, \rho} > 0}$, depending on $\gamma$ and on the mollifier $\rho$, such that the sequence $\smash{(\mu^{\star}_{\gamma, \ceps})_{\eps > 0}}$ defined in \eqref{e:seq_super_main_star} converges $\sigma(\rmX^{\star})$-stably to $\smash{a_{\gamma, \rho}^{\star} \CP_{\gamma}[\mu^{\star}_{\gammac}]}$ along any sequence $(\eps_n)_{n \in \N}$ converging to $0$, where $\mu_{\gammac}^{\star}$ denotes the critical GMC associated with $\rmX^{\star}$.
\end{lemma}
\begin{proof}
We divide the proof into three main parts. In the first step, we set up the framework of the proof. In the second step, we show how the proof of the convergence of $\smash{\mu^{\star}_{\gamma, \ceps}}$ can be reduced to the convergence of the measure $\smash{\bar \mu_{\gamma, t}}$, defined in \eqref{eq:defBarMuGammaTn}, which is essentially the normalised supercritical GMC constructed from the field $\rmX_{t}^\star$ plus $\rmW_{t}$. Finally, in the last step, by referencing Lemma~\ref{lm:stableConvStartn} \dash stated in Section~\ref{sec:RemovingCompact} and whose proof is based on \cite[Theorem~C]{BHCluster} \dash we establish the convergence of $\smash{\bar \mu_{\gamma, t}}$.

\textbf{Step~1:} In this step, we establish the framework for the remainder of the proof and introduce several definitions. We begin by recalling that, by Proposition~\ref{pr:decoCovStar}, there exists a constant $\frka = \frka(\rho, \frkK) \in (0, 1)$ and a smooth, stationary, centred Gaussian field $\rmW$, with decaying correlations and independent of $\CF_{\infty}$ (recall \eqref{eq:filtrationStar}), such that for any fixed $\eps > 0$, defining $\smash{t_{\eps} \eqdef \log(\frka \eps^{-1})}$, it holds that  
\begin{equation*}
\rmX^{\star}_{\ceps} \eqlaw \rmX^{\star}_{t_{\eps}} + \rmW_{t_{\eps}} + \rmZ_{t_{\eps}} \;,
\end{equation*}
where we recall that all the fields on the right-hand side are mutually independent, $\rmW_{t_{\eps}}(\cdot) = \rmW(e^{t_{\eps}} \cdot)$, and $\rmZ_{t_{\eps}}$ is a smooth, stationary, centred Gaussian field with vanishing variance. 

We fix a sequence $(\eps_n)_{n \in \N}$ converging to $0$ and we write $t_n$ instead of $t_{\eps_n}$. By possibly enlarging the probability space on which $\rmX^{\star}$ is defined, we can (and will) assume that the following fields are also defined on the same probability space:
\begin{enumerate}[start=1,label={{{(S\arabic*})}}]
	\item \label{st_1} A copy of the field $\rmW$ independent of $\CF_{\infty}$.
	\item \label{st_2} A collection of fields $(\rmZ_{t_n})_{n \in \N}$, all mutually independent and  independent of everything else, such that $\rmZ_{t_n} \eqlaw \rmZ_{t_{\eps_n}}$. 
	\item \label{st_3} A collection of fields $(\rmX^{\star, n})_{n \in \N}$ where each $\rmX^{\star, n}$ has the same law as $\rmX^{\star}$, and the conditional law of $\rmX^{\star, n}$ given $\rmX^{\star}_{t_n} + \rmW_{t_n} + \rmZ_n$ coincides with the conditional law of $\rmX^{\star}$ given $\rmX^{\star}_{(\eps_n)}$. 
\end{enumerate}
We observe that \ref{st_3} guarantees that one has the almost sure identity
\begin{equation*}
\rmX^{\star,n}_{\cepsn} = \rmX^{\star}_{t_n} + \rmW_{t_n} + \rmZ_{t_n}\;.
\end{equation*}
For each $n \in \N$, we define the measure $\smash{\mu^{\star}_{\gamma, t_n}}$ by
\begin{equation*}
\mu^{\star}_{\gamma, t_n}(dx) \eqdef	 t_n^{\f{3\gamma}{2\sqrt{\smash[b]{2d}}}} e^{-t_n(\gamma/\sqrt 2-\sqrt{d})^2} e^{\gamma (\rmX^{\star}_{t_{n}}(x) + \rmW_{t_{n}}(x)) - \frac{\gamma^2}{2} t_{n}} dx \;, 
\end{equation*}
Furthermore, we define the random function $\frkm_{t_n}: \R^d \to \R$ by setting
\begin{equation}
\label{eq:defFrkMn}
\frkm_{t_n}(x) \eqdef c_{n} \frka^{-(\gamma/\sqrt{\smash[b]{2}} - \sqrt{\smash[b]{d}})^2} e^{\gamma \rmZ_{t_{n}}(x) - \f{\gamma^2}{2} \E[\rmW_{t_{n}}(x)^2+\rmZ_{t_{n}}(x)^2]} \;,
\end{equation}
and we define the measure $\bar \mu_{\gamma, t_n}$ given by
\begin{equation}
\label{eq:defBarMuGammaTn}
\bar \mu^{\star}_{\gamma, t_n}(dx) \eqdef \frkm_{t_n}(x) \mu^{\star}_{\gamma, t_n}(dx)	\;.
\end{equation}
The constant $c_n$ in \eqref{eq:defFrkMn} is the constant converging to $1$ as $n \to \infty$ such that $\bar \mu^{\star}_{\gamma, t_n} = \mu^{\star, n}_{\gamma, \cepsn}$ almost surely, where $\smash{\mu^{\star, n}_{\gamma, \cepsn}}$ denotes the measure defined in \eqref{e:seq_super_main_star} with $\smash{\rmX^{\star}_{\cepsn}}$ replaced by $\smash{\rmX^{\star,n}_{\cepsn}}$.

\textbf{Step~2:} With the setup and notation introduced above, to show that $\smash{\mu^{\star}_{\gamma, \cepsn}}$ converges $\sigma(\rmX^{\star})$-stably, thanks to Lemma~\ref{lm:stableRV} and \cite[Lemma~3.4]{BHCluster}, it suffices to check that for all $\smash{(\varphi, f) \in \CC^{\infty}_c(\R^d) \times \CC^{+}_c(\R^d)}$, it holds that 
\begin{equation} 
\label{eq:needProveStable1New}
\lim_{n \to \infty} \E\Bigl[\exp\bigl(i \langle \rmX^{\star}, \varphi\rangle\bigr) \exp\bigl(-\mu^{\star}_{\gamma, \cepsn}(f)\bigr)\Bigr] = \E\Bigl[\exp\bigl(i \langle \rmX^{\star}, \varphi\rangle\bigr) \exp\bigl(-\mu^{\star}_{\gamma}(f)\bigr)\Bigr] \;,
\end{equation}
where we set $\smash{\mu^{\star}_{\gamma} \eqdef a_{\gamma, \rho}^{\star} \CP_{\gamma}[\mu^{\star}_{\gammac}]}$, with $\smash{a_{\gamma, \rho}^{\star} > 0}$ the constant specified in the statement. 
We observe that for each $n \in \N$, the following equality holds thanks to \ref{st_3}:
\begin{equation*}
\E\Bigl[\exp\bigl(i \langle \rmX^{\star}, \varphi\rangle\bigr) \exp\bigl(-\mu^{\star}_{\gamma, \cepsn}(f)\bigr)\Bigr] = \E\Bigl[\exp\bigl(i \langle \rmX^{\star, n}, \varphi\rangle\bigr) \exp\bigl(-\bar \mu^{\star}_{\gamma, t_n}(f)\bigr)\Bigr] \;.
\end{equation*}
Now, we proceed to show that 
\begin{equation}
\label{eq:needProveStable2New}
\lim_{n \to \infty} \Bigl|\E\Bigl[\exp\bigl(i \langle \rmX^{\star, n}, \varphi\rangle\bigr) \exp\bigl(-\bar \mu^{\star}_{\gamma, t_n}(f)\bigr)\Bigr] - \E\Bigl[\exp\bigl(i \langle \rmX^{\star}, \varphi\rangle\bigr) \exp\bigl(-\bar \mu^{\star}_{\gamma, t_n}(f)\bigr)\Bigr]\Bigr| = 0 \;.
\end{equation}
The proof of \eqref{eq:needProveStable2New} proceeds in three steps. First, thanks to \ref{st_3} and the Cauchy--Schwarz inequality, we obtain that
\begin{align*}
\E\Bigl[\Bigl|\exp\bigl(i \langle \rmX^{\star, n}, \varphi\rangle\bigr) - \exp\bigl(i \langle \rmX^{\star, n}_{\cepsn}, \varphi\rangle\bigr)\Bigr|\Bigr] 
& = \E\Bigl[\Bigl|\exp\bigl(i \langle \rmX^{\star}, \varphi\rangle\bigr) - \exp\bigl(i \langle \rmX^{\star}_{\cepsn}, \varphi\rangle\bigr) \Bigr|\Bigr] \\
& \leq \E\Bigl[\Bigl|\langle \rmX^{\star}, \varphi\rangle - \langle \rmX^{\star}_{\cepsn}, \varphi\rangle \Bigr|^2 \Bigr]^{1/2} \;,
\end{align*}
and, as one can easily check, the quantity on the second line of the above display converges to $0$ as $n \to \infty$. Next, thanks again to \ref{st_3} and the Cauchy--Schwarz inequality, we have that
\begin{align*}
\E\Bigl[\Bigl|\exp\bigl(i \langle \rmX^{\star, n}_{\cepsn}, \varphi\rangle\bigr) - \exp\bigl(i \langle \rmX^{\star}_{t_n}, \varphi\rangle\bigr)\Bigr|\Bigr] 
& = \E\Bigl[\Bigl|\exp\bigl(i \langle \rmX^{\star}_{t_n} + \rmW_{t_n} + \rmZ_{t_n}, \varphi\rangle\bigr) - \exp\bigl(i \langle \rmX^{\star}_{t_n}, \varphi\rangle\bigr)\Bigr|\Bigr] \\
& \leq \E\Bigl[\Bigl|\langle \rmW_{t_n} + \rmZ_{t_n}, \varphi\rangle\Bigr|^2\Bigr]^{1/2} \;,
\end{align*}	
and again the quantity on the second line of the above display converges to $0$ as $n \to \infty$. This follows from the fact that $\rmW_{t_n}(\cdot) = \rmW(e^{t_n} \cdot)$, that $\rmW$ has decaying correlations, and from the limit $\lim_{n \to \infty} \E[\rmZ_{t_n}(\cdot)^2] = 0$.
Finally, to complete the proof of \eqref{eq:needProveStable2New}, it remains to show that
\begin{equation*}
\lim_{n \to \infty} \E\Bigl[\Bigl|\exp\bigl(i \langle \rmX^{\star}_{t_n}, \varphi\rangle\bigr) - \exp\bigl(i \langle \rmX^{\star}, \varphi\rangle\bigr)\Bigr|\Bigr] = 0 \;,
\end{equation*}
which follows immediately from the almost sure convergence of $\rmX^{\star}_{t_n}$ to $\rmX^{\star}$ in $\CH_{\loc}^{-\kappa}(\R^d)$ as $n \to \infty$, for any $\kappa > 0$.

\textbf{Step~3:} Having proved \eqref{eq:needProveStable2New}, to finish the proof, it remains to check that
\begin{equation} 
\label{eq:needProveStable3New}
\lim_{n \to \infty} \E\Bigl[\exp\bigl(i \langle \rmX^{\star}, \varphi\rangle\bigr) \exp\bigl(-\bar \mu^{\star}_{\gamma, t_n}(f)\bigr)\Bigr] = \E\Bigl[\exp\bigl(i \langle \rmX^{\star}, \varphi\rangle\bigr) \exp\bigl(-\mu^{\star}_{\gamma}(f)\bigr)\Bigr] \;,
\end{equation}
which is equivalent to the fact that  $\bar\mu^{\star}_{\gamma, t_n}$ converges $\sigma(\rmX^{\star})$-stably to $\mu^{\star}_{\gamma}$. We recall that
\begin{equation*}
\bar \mu^{\star}_{\gamma, t_n}(dx) = \frkm_{t_n}(x) \mu^{\star}_{\gamma, t_{n}} (dx) \;, 
\end{equation*}
where $\frkm_{t_n}$ is the function defined in \eqref{eq:defFrkMn}.
Thanks to Lemma~\ref{lm:stableConvStartn}, which is stated and proved in Section~\ref{sec:RemovingCompact} and whose proof is based on \cite[Theorem~C]{BHCluster}, we know that $\mu^{\star}_{\gamma, t_{n}}$ converges $\sigma(\rmX^{\star})$-stably to $c^{\star}_{\gamma, \rho} \CP_{\gamma}[\mu^{\star}_{\gammac}]$ for some finite constant $c^{\star}_{\gamma, \rho} > 0$, depending on $\gamma$ and the mollifier $\rho$\footnote{The dependence of the constant $c^{\star}_{\gamma, \rho}$ on the mollifier $\rho$ arises from the fact that the constant $c^{\star}_{\gamma}$ in Lemma~\ref{lm:stableConvStartn} implicitly depends on the law of the field $\rmW$.}. On the other hand, we have that $\E[\rmW_{t_{n}}(\cdot)^2]$ is a finite constant since the field $\rmW$ is stationary. Moreover, we have that $\lim_{n \to \infty}\E[\rmZ_{t_{n}}(\cdot)^2] = 0$ on $\R^d$. Hence, combining these two facts, we deduce that the random function $\frkm_{t_n}$ converges in probability to a constant $\frkm$ as $n \to \infty$, with respect to the topology of local uniform convergence in $\CC(\R^d)$. Therefore, thanks to Lemma~\ref{lm:SlutskyStable}, the desired result follows with $a_{\gamma, \rho}^{\star} = \frkm \, c^{\star}_{\gamma, \rho}$.
\end{proof}

With Lemma~\ref{lm:stableConvStar} established, the proof of Theorem~\ref{th:convergence} follows from the decomposition \eqref{eq:compConvMol} and from the properties of stable convergence.

\begin{proof}[Proof of Theorem~\ref{th:convergence}]
Consider the setting introduced at the beginning of this section, and fix a sequence $(\eps_n)_{n \in \N}$ that converges to $0$.
 We recall that we need to prove that $\mu_{\gamma, \cepsn}$ converges $\sigma(\rmX)$-stably as $n \to \infty$ to $\smash{\mu_{\gamma} \eqdef a_{\gamma, \rho} \CP_{\gamma}[\bar{\mu}_{\gammac}]}$. 
 Since $\CM^{+}(\rmD)$ is endowed with the topology of vague convergence, then $\mu_{\gamma, \cepsn}$ converges $\sigma(\rmX)$-stably to $\mu_{\gamma}$ if and only if $\mu_{\gamma, \cepsn}|_{\rmA}$ converges $\sigma(\rmX)$-stably to $\mu_{\gamma}|_{\rmA}$, for all compact subsets $\rmA \subset \rmD$. Here, $\mu_{\gamma, \cepsn}|_{\rmA}$ (resp.\ $\mu_{\gamma}|_{\rmA}$) denotes the restriction of the measures $\mu_{\gamma, \cepsn}$ (resp.\ $\mu_{\gamma}$) to $\rmA$. Hence, in what follows, we fix a compact subset $\rmA \subset \rmD$, and in order to lighten the notation, we simply write $\mu_{\gamma, \cepsn}$ (resp.\ $\mu_{\gamma}$) instead of $\mu_{\gamma, \cepsn}|_{\rmA}$ (resp.\ $\mu_{\gamma}|_{\rmA}$).
 Before proceeding, we recall that, thanks to \eqref{eq:compConvMol}, it suffices to show the convergence of $\frkh_{\gamma, \cepsn} \mu^{\star}_{\gamma, \cepsn}(dx)$.

\textbf{Step~1}.
Recalling definition \eqref{eq:smoothBit}, we note that the sequence $(\frkh_{\gamma, \cepsn})_{\eps > 0}$ can be viewed as a collection of random functions in the space $\CC(\rmA)$. In particular, there exists a random function $\frkh_{\gamma} \in \CC(\rmA)$ such that $\frkh_{\gamma, \cepsn}$ converges as $n \to \infty$ to $\frkh_{\gamma}$ in probability with respect to the topology of uniform convergence in $\CC(\rmA)$. Indeed, we have that $\rmL_{\cepsn}$ converges as $n \to \infty$ to $\rmL$ in probability in $\CC(\rmA)$, and $\E[\rmX_{\cepsn}(x)^2 - {\rmX^{\star}_{\cepsn}(x)}^2]$ converges uniformly in $\rmA$ as $n \to \infty$. This is due to the fact that the covariance kernels of both $\rmX$ and $\rmX^{\star}$ can be written as the sum of $-\log|\cdot - \cdot|$ and some H\"older continuous functions $g$, $g^{\star} : \rmA \times  \rmA \to \R$. In particular, this implies that 
\begin{equation*}
\lim_{n \to \infty} \sup_{x \in \rmA} \E\bigl[\rmX_{\cepsn}(x)^2 - {\rmX^{\star}_{\cepsn}(x)}^2\bigr] =  (g - g^{\star})(x, x) \;.
\end{equation*}
Therefore, putting everything together, the random function $\frkh_{\gamma} : \rmA \to \R$ can be written as follows
\begin{equation}
\label{eq:expressionFrkh}
\frkh_{\gamma}(x) = e^{\gamma \rmL(x) - \frac{\gamma^2}{2}(g - g^{\star})(x, x)} \;.
\end{equation}

\textbf{Step~2}.
Thanks to Lemma~\ref{lm:stableConvStar}, we know that the sequence of random measures $\smash{(\mu^{\star}_{\gamma, \cepsn})_{n \in \N}}$ converges $\smash{\sigma(\rmX^{\star})}$-stably as $n \to \infty$ to $\smash{a_{\gamma, \rho}^{\star} \CP_{\gamma}[\mu^{\star}_{\gammac}]}$, where $\mu_{\gammac}^{\star}$ denotes the critical GMC associated with $\rmX^{\star}$. In particular, since $\smash{\mu^{\star}_{\gamma, \cepsn}}$ is conditionally independent of $\rmL$ given $\rmX^{\star}$, thanks to Lemma~\ref{lm:doubleStable}, we have that $\smash{\mu^{\star}_{\gamma, \cepsn}}$ converges $\sigma(\rmX^{\star}\!, \rmL)$-stably as $n \to \infty$ to $\smash{a_{\gamma, \rho}^{\star} \CP_{\gamma}[\mu^{\star}_{\gammac}]}$. Furthermore, thanks to the previous step, we know that $\frkh_{\gamma, \cepsn}$ converges as $n \to \infty$ to $\frkh_{\gamma}$ in probability with respect to the topology of uniform convergence in $\CC(\rmA)$. Moreover, the random function $\frkh_{\gamma}$ is $\sigma(\rmX^{\star}, \rmL)$-measurable, and so, thanks to Lemma~\ref{lm:SlutskyStable}, we have that $\smash{\frkh_{\gamma, \cepsn} \mu^{\star}_{\gamma, \cepsn}}$ converges $\sigma(\rmX^{\star}\!, \rmL)$-stably as $n \to \infty$ to $\smash{a_{\gamma, \rho}^{\star} \frkh_{\gamma} \CP_{\gamma}[\mu^{\star}_{\gammac}]}$. Since $\sigma(\rmX) \subseteq \sigma(\rmX^{\star}\!, \rmL)$, this implies that $\smash{\frkh_{\gamma, \cepsn} \mu^{\star}_{\gamma, \cepsn}}$ converges $\smash{\sigma(\rmX)}$-stably as $n \to \infty$ to $\smash{a_{\gamma, \rho}^{\star}  \frkh_{\gamma} \CP_{\gamma}[\mu^{\star}_{\gammac}]}$.

\textbf{Step~3}. To conclude, it suffices to check that there exists a constant $a_{\gamma, \rho} > 0$, depending on $\gamma$ and on the mollifier $\rho$, such that 
\begin{equation}
\label{eq:concMainTh}
a_{\gamma, \rho}^{\star} \frkh_{\gamma} \CP_{\gamma}[\mu^{\star}_{\gammac}]  = a_{\gamma, \rho}  \CP_{\gamma}[\bar\mu_{\gammac}] \;,
\end{equation}
where $\bar\mu_{\gammac}$ is the measure defined in \eqref{eq:tiltedCritical}. 
To verify that \eqref{eq:concMainTh} holds, we begin by observing that, thanks to Lemma~\ref{lm:doubleStable}, $\frkh_{\gamma}$ is a random non-negative continuous function that is conditionally independent of $\CP_{\gamma}[\mu^{\star}_{\gammac}]$ given $\rmX^{\star}$. 
Therefore, thanks to Lemma \ref{lem:mult}, it holds that
\begin{equation*}
a_{\gamma, \rho}^{\star}  \frkh_{\gamma} \CP_{\gamma}\bigl[\mu^{\star}_{\gammac}\bigr] = a_{\gamma, \rho}^{\star} \CP_{\gamma}\bigl[\frkh_{\gamma}^{\sqrt{\smash[b]{2d}}/\gamma} \mu^{\star}_{\gammac}\bigr] \;.
\end{equation*}
On the other hand, arguing in a similar way as in Step~1 (see also the proof of \cite[Theorem~5.4]{Junnila_deco}), we have that 
\begin{equation*}
	\mu_{\gammac}(dx) = e^{\sqrt{\smash[b]{2d}} \,\rmL(x) - d (g - g^{\star})(x, x)} \mu^{\star}_{\gammac}(dx) \;,
\end{equation*}
where we recall that $\mu_{\gammac}$ (resp.\ $\mu^{\star}_{\gammac}$) denotes the critical GMC associated with $\rmX$ (resp.\ $\rmX^{\star}$). 
Therefore, recalling the expression \eqref{eq:expressionFrkh} for the random function $\frkh_{\gamma}$, it holds that
\begin{equation*}
\frkh_{\gamma}(x)^{\sqrt{\smash[b]{2d}}/\gamma} \mu^{\star}_{\gammac}(dx) = e^{(d - \sqrt{\smash[b]{d/2}}\gamma) (g - g^{\star})(x, x)} \mu_{\gammac}(dx) \;.
\end{equation*}
Since by \ref{hp_K1} the seed covariance function $\frkK$ is radial, it holds that the function $g^{\star}$ is constant along the diagonal. Hence, this allows us to conclude that there exists a finite constant $b_{\star} > 0$ such that 
\begin{equation*}
a_{\gamma, \rho}^{\star}  \CP_{\gamma}\bigl[\frkh_{\gamma}^{\sqrt{\smash[b]{2d}}/\gamma} \mu^{\star}_{\gammac}\bigr] = a_{\gamma, \rho}^{\star}  \CP_{\gamma}\bigl[b_{\star} \bar\mu_{\gammac}\bigr] \;.
\end{equation*}
Therefore, the conclusion then follows by factoring out the constant $b_{\star}$ by using Lemma~\ref{lem:mult}.
\end{proof}
 
\subsection{Proof of Proposition~\ref{pr:decoCovStar}}
\label{sub:conv}
\begin{proof}
For all $x \in \R^d$, we set
\begin{equation*}
\CK^{\rho}_{\ceps}(x) \eqdef \E\bigl[\rmX^{\star}_{\ceps}(x)\rmX^{\star}_{\ceps}(0)\bigr] \;, \qquad \CK_{t}(x) \eqdef \E\bigl[\rmX^{\star}_{t}(x)\rmX^{\star}_{t}(0)\bigr] \;.
\end{equation*}
In order to prove the result, it suffices to prove that the difference between the Fourier transforms of $\smash{\CK^{\rho}_{\ceps}}$ and $\smash{\CK_{t}}$ can be written as the sum of two non-negative functions satisfying some suitable properties. 

We start by fixing a mollifier $\rho$ satisfying assumptions \ref{hp_A1} \dash \ref{hp_A2}. Since $\rho \in \CC^{\infty}_c(\R^d)$ has unit mass, we have that its Fourier transform $\hat \rho$ is a smooth and rapidly decaying function satisfying $\hat \rho(0) = 1$. 
We also fix a $\star$-scale invariant field $\rmX^{\star}$ on $\R^d$ whose seed covariance function $\frkK$ satisfies assumptions \ref{hp_K1} \dash \ref{hp_K3}. 
We recall that its Fourier transform $\smash{\hat \frkK}$ is non-negative, radial, and supported in $B(0,1)$.

\textbf{Step~1}.
We recall that, for $t \geq 0$ and $\eps > 0$, we have that, for all $x \in \R^d$, 
\begin{equation*}
\CK^{\rho}_{\ceps}(x) = \int_0^{\infty}  \bigl(\rho_{\eps} * \bar{\rho}_{\eps} * \frkK(e^u \cdot)\bigr) (x) du\;, \qquad \CK_{t}(x) = \int_0^{t} \frkK(e^u x) du\;,
\end{equation*}
where $\smash{\bar\rho_{\eps}(\cdot) \eqdef \rho_{\eps}(-\cdot)}$. 
We now fix a positive constant $\frka \in (0, 1)$ to be determined later, and for all $t \geq 0$, we set $\smash{\CK^{\rho}_{t}(x) \eqdef \CK^{\rho}_{(\frka e^{-t})}(x)}$. The Fourier transforms of $\CK^{\rho}_{t}$ and $\CK_{t}$ are given for all $\omega \in \R^d$ by 
\begin{equation*}
\hat \CK_t^\rho(\omega) = \abs{\hat\rho(\frka e^{-t}\omega)}^2 \int_0^\infty \hat \frkK(e^{-u} \omega)e^{-du} du\;, \qquad  \hat \CK_t(\omega) = \int_0^t \hat \frkK(e^{-u} \omega)e^{-du} du\;.
\end{equation*}
Now, we consider the function $\hat \Delta^{\rho}_t : \R^d \to \R$ given by $\hat \Delta^{\rho}_t(\omega) \eqdef \hat \CK_t^\rho(\omega) - \hat \CK_t(\omega)$. A simple computation yields that 
\begin{equation*}
\hat \Delta^{\rho}_t(\omega) = e^{-dt}\hat \CK(e^{-t}\omega) - \bigl(1 - \abs{\hat\rho(\frka e^{-t}\omega)}^2\bigr)\hat \CK(\omega)\;, \qquad \forall \, \omega \in \R^d \;,
\end{equation*}
where $\smash{\hat \CK \eqdef \hat \CK_{\infty}}$. Now, letting $e_1$ be the first unit vector of $\R^d$, we can write for all $\omega \neq 0$,
\begin{align*}
	\hat \CK(\omega)  = \int_0^{\infty} \hat \frkK \big(e^{-u} \omega \big) e^{- d u} d u & = |\omega|^{-d} \int_0^{|\omega|} \hat \frkK(u e_1) u^{d-1} du  = |\mathbb{S}^{d-1}|^{-1} |\omega|^{-d} \int_{|\xi| \leq |\omega|} \hat{\frkK}(\xi) d \xi \;,
\end{align*}
where $|\mathbb{S}^{d-1}|$ denotes the $(d-1)$-Lebesgue measure of the $(d-1)$-dimensional unit sphere. Therefore, if we let $\CT:\R^d \to \R$ be the function given by 
\begin{equation*}
\CT(\omega) \eqdef \int_{|\xi| \leq |\omega|} \hat{\frkK}(\xi) d \xi \;, \qquad \forall \, \omega \in \R^d \;,
\end{equation*}
we can then write 
\begin{equation*}
\hat \CK(\omega) = |\mathbb{S}^{d-1}|^{-1} |\omega|^{-d} \CT(\omega)\;, \qquad \forall \, \omega \in \R^d \setminus \{0\}\;.
\end{equation*}
Furthermore, since $\frkK$ is positive definite and satisfies $\frkK(0) = 1$ (implying that $\smash{\hat{\frkK}} \geq 0$ and that $\smash{\hat{\frkK}}$ has unit mass), and given that $\smash{\hat{\frkK}}$ is supported in $B(0, 1)$, we observe that the function $\CT$ satisfies the following properties:
\begin{equation*}
\CT(\omega) \in [0, 1] \;, \; \forall \, \omega \in \R^d, \; \qquad \CT(\omega) = 1 \;, \; \forall \, \abs{\omega} \geq 1, \; \qquad \int_{\abs{\omega} < 1}\abs{\omega}^{-d} \CT(\omega) d \omega < \infty \;,
\end{equation*}
where the third property follows from \ref{hp_K3}.
With this notation in hand, we can rewrite $\hat \Delta_t^{\rho}$ as follows
\begin{equation*}
\hat \Delta_t^\rho(\omega) = |\mathbb{S}^{d-1}|^{-1} |\omega|^{-d} \bigl(\CT(e^{-t}\omega) - \bigl(1-\abs{\hat\rho(\frka e^{-t}\omega)}^2\bigr)\CT(\omega)\bigr)\;, \qquad \forall \, \omega \in \R^d \setminus \{0\}\;.
\end{equation*}

\textbf{Step~2}.
Now, we consider the function $\hat{\CK}_{\rmW} :\R^d \to \R$ given by
\begin{equation}
\label{e:hatKW}
	\hat{\CK}_{\rmW}(\omega) \eqdef |\mathbb{S}^{d-1}|^{-1}|\omega|^{-d} \left(\CT(\omega) - \bigl(1- \abs{\hat\rho(\frka \omega)}^2\bigr)\right) \;, \qquad \forall \, \omega \in \R^d \setminus \{0\} \;,
\end{equation}
which is an integrable function. In particular, we observe that the integrability at infinity follows from the rapid decay of $\hat{\rho}$ and the fact that $\CT(\omega) = 1$ for all $\abs{\omega} \geq 1$. The integrability near the origin follows from the properties of the function $\CT$ listed in the previous step, and from the fact that $\hat \rho$ is smooth and $\hat \rho(0) = 1$.

Furthermore, we claim that we can find a constant $\frka = \frka(\rho, \frkK) \in (0, 1)$ such that the function $\smash{\hat{\CK}_{\rmW}}$ is non-negative on $\R^d$. To see this, we need to split into two cases. 
\begin{enumerate}
	\item If $|\omega| \geq 1$, the fact that $\hat{\CK}_{\rmW}(\omega)$ is non-negative follows simply by noticing that $\CT(\omega) = 1$ for all $|\omega| \geq 1$. Therefore, the term in the brackets in the definition \eqref{e:hatKW} of $\hat{\CK}_{\rmW}$ is just equal to $\abs{\hat\rho(\frka \omega)}^2$ which is obviously non-negative.
	\item If $|\omega| < 1$, we need to exploit the fact that $\frka \in (0, 1)$ can be chosen small. Thanks to assumption \ref{hp_K3}, we have that $\smash{\CT(\omega) \geq \bar \frka |\omega|^d}$, for all $|\omega| < 1$ and for some constant $\bar \frka > 0$. On the other hand, by assumption~\ref{hp_A2}, there exists $\zeta > 0$, only depending on $\hat \rho$, such that
\begin{equation*}
\abs{\hat \rho(\omega)}^2 \geq 1 - 2\abs{\omega}^d \Biggl(\sum_{\rmj \in \N_0^d, \, \abs{\rmj} = d} \abs{\partial^{\rmj} \hat \rho(0)}/ \rmj! + 1/10\Biggr)\;, \quad \forall \,|\omega| < \zeta\;,
\end{equation*}
Hence, putting everything together and choosing $\frka < \zeta$, we obtain that
\begin{equation*}
\CT(\omega) - \bigl(1 - \abs{\hat\rho(\frka \omega)}^2\bigr) \geq |\omega|^d \Biggl(\bar \frka - 2 \frka^d\Biggl(\sum_{\rmj \in \N_0^d, \, \abs{\rmj} = d} \abs{\partial^{\rmj} \hat \rho(0)} / \rmj! + 1/10\Biggr)\Biggr) \;, \quad \forall \, |\omega| < 1\;, 
\end{equation*}
which is positive as long as we choose $\frka > 0$ small enough. 
\end{enumerate}
Finally, recalling once again that $\CT(\omega) = 1$ for all $\abs{\omega} \geq 1$, we observe that the rapid decay of the function $\hat{\rho}$ implies that $\hat{\CK}_{\rmW}$ is also rapidly decaying.

\textbf{Step~3}.
Now, for $\frka \in (0, 1)$ as specified in the previous step, we introduce the function $\hat \CK_{\rmZ,t}: \R^d \to \R$ defined as follows
\begin{equation*}
\hat{\CK}_{\rmZ, t}(\omega) =  |\mathbb{S}^{d-1}|^{-1} |\omega|^{-d} \bigl(1 - \abs{\hat\rho(\frka e^{-t} \omega)}^2\bigr)\bigl(1-\CT(\omega)\bigr) \;, \qquad \forall \, \omega \in \R^d \setminus \{0\} \;,
\end{equation*}
which is an integrable function for a similar reason as explained in Step~2.
We recall that $\hat{\rho}(0) = 1$, and that, by assumption~\ref{hp_A2}, $\abs{\hat{\rho}}$ attains a local maximum at $0$. Therefore, by taking $\frka \in (0, 1)$ even smaller if necessary, we have that $\smash{1 - \abs{\hat\rho(\frka e^{-t} \omega)}^2 \geq 0}$ for all $\abs{\omega} \leq 1$ and $t \geq 0$. 
Since the function $\CT$ takes values in $[0, 1]$ and satisfies $\CT(\omega) = 1$ for all $|\omega| > 1$, it follows that 
the function $\hat{\CK}_{\rmZ,t}$ is non-negative and compactly supported in the ball $B(0,1)$. 
Finally, since $\hat{\rho}(0) = 1$, we have that $\hat{\CK}_{\rmZ, t}(\omega)\to 0$ as $t \to \infty$, uniformly over all $\omega \in \R^d$.

\textbf{Step~4}.
For each $t \geq 0$, we define the function $\hat \CK_{\rmW, t} : \R^d \to \R$ by letting $\smash{\hat \CK_{\rmW,t}(\omega) = e^{-dt} \hat{\CK}_{\rmW}(e^{-t} \omega)}$, where $\smash{\hat \CK_{\rmW}}$ is the function defined in \eqref{e:hatKW}. Then, as one can easily verify, we can write the function $\hat \Delta_t^{\rho}$ as the following sum
\begin{equation}
\label{eq:relationDiffKernFourier}
\hat \Delta_t^\rho(\omega) = \hat \CK_{\rmW, t}(\omega) + \hat \CK_{\rmZ, t}(\omega)\;, \quad \forall \, \omega \in \R^d \setminus \{0\} \;.
\end{equation}
Thanks to the Steps~2 and~3, both terms on the right-hand side of the above expression are non-negative for a suitable choice of the constant $\frka \in (0, 1)$. 
Letting $\smash{\CK_{\rmW,t}}$ be the inverse Fourier transform of $\smash{\hat \CK_{\rmW,t}}$, it is easily seen that $\smash{\CK_{\rmW,t}(\cdot) = \CK_{\rmW}(e^t \cdot)}$, where $\CK_{\rmW}$ is the inverse Fourier transform of $\smash{\hat \CK_{\rmW}}$. 
Moreover, $\CK_{\rmW}$ is a smooth, decaying function since its Fourier transform is integrable and has rapid decay.
Furthermore, the function $\CK_{\rmW}$ is even. Therefore, by Bochner's theorem, it serves as the covariance kernel of a smooth, stationary, centred Gaussian field $\rmW$ with decaying correlations.

Similarly, we let $\smash{\CK_{\rmZ,t}}$ be the inverse Fourier transform of $\smash{\hat \CK_{\rmZ,t}}$. We note that $\smash{\CK_{\rmZ,t}}$ is smooth since its Fourier transform is compactly supported. 
Furthermore, since $\hat{\CK}_{\rmZ, t}(\omega)\to 0$ as $t \to \infty$ for all $\omega \in \R^d \setminus \{0\}$, by the dominated convergence theorem, we have that $\CK_{\rmZ, t}(0) \to 0$ as $t \to \infty$.
Moreover, for each $t \geq 0$, the function $\CK_{\rmZ, t}$ is even. Hence, by Bochner's theorem once again, it is the covariance kernel of a smooth, stationary, centred Gaussian field $\rmZ_t$, satisfying $\lim_{t \to \infty} \E[\rmZ_t(\cdot)^2] = 0$.

To conclude, we observe that the decomposition \eqref{e:decoCovStar} follows directly from \eqref{eq:relationDiffKernFourier} by taking the inverse Fourier transform.
\end{proof}

\section{The case of infinite-range correlations}
\label{sec:RemovingCompact}
The main goal of this section is to extend the convergence result in \cite[Theorem~C]{BHCluster} to the setting in which the fields may exhibit infinite-range correlations.
To be more precise, the setting we consider is as follows. Let $\rmX^{\star}$ be a $\star$-scale invariant field with seed covariance function $\frkK$ satisfying assumptions \ref{hp_K1} \dash \ref{hp_K2}, and denote by $(\rmX^{\star}_t)_{t \geq 0}$ its martingale approximation.
 Note that in this section, we drop assumption~\ref{hp_K3}.
We also consider a smooth, stationary, centered Gaussian field $\rmW$ with decaying correlations, independent of $\CF_{\infty}$ (recall \eqref{eq:filtrationStar}). For each $t \geq 0$, we define $\rmW_t(\cdot) \eqdef \rmW(e^t \cdot)$.
We define the sequence of measures $(\mu^{\star}_{\gamma, t})_{t \geq 0}$ as follows
\begin{equation}
\label{e:seq_super_main_star_ttt}
\mu^{\star}_{\gamma, t}(dx) \eqdef t^{\f{3\gamma}{2\sqrt{\smash[b]{2d}}}} e^{t(\gamma/\sqrt 2-\sqrt{d})^2} e^{\gamma (\rmX^{\star}_{t}(x) + \rmW_{t}(x)) - \frac{\gamma^2}{2} t} dx\;.
\end{equation}

With this notation in place, we are now ready to state the main result of this section.
\begin{lemma}
\label{lm:stableConvStartn}
For $\gamma > \sqrt{\smash[b]{2d}}$, there exists a finite constant  $\smash{c^{\star}_{\gamma} > 0}$ such that the sequence $\smash{(\mu^{\star}_{\gamma, t})_{t > 0}}$ defined in \eqref{e:seq_super_main_star_ttt} converges $\sigma(\rmX^{\star})$-stably to $\smash{c_{\gamma}^{\star} \CP_{\gamma}[\mu^{\star}_{\gammac}]}$ as $t \to \infty$, where $\mu_{\gammac}^{\star}$ denotes the critical GMC associated with $\rmX^{\star}$.
\end{lemma}
	
We begin by noting that the convergence of the measure $\mu^{\star}_{\gamma, t}$ would follow directly from \cite[Theorem~C]{BHCluster}, provided that the seed covariance function $\frkK$ of the field $\rmX^{\star}$ and the covariance of the field $\rmW$ are compactly supported. Therefore, the goal of this section is to extend this convergence also to the case where $\frkK$ is not compactly supported and $\rmW$ satisfies the assumptions mentioned above.
 The strategy first involves introducing a collection of fields that approximate $\rmX^{\star}$ and $\rmW$ with finite-range correlations. Then, by removing the cutoff and applying Kahane's convexity inequality, we show how we can obtain the desired result. 

We recall that Kahane's convexity inequality essentially allows for the comparison of GMC measures associated with two slightly different fields. 
It can be stated as follows, and we refer to \cite[Theorem~3.18]{BP24} or \cite[Theorem~2.1]{RV_Review} for a proof and additional references.
\begin{lemma}
\label{lm_Kahane}
Consider a bounded domain $\rmD$ and two almost surely continuous centred Gaussian fields $(\rmX(x))_{x \in \rmD}$ and $(\rmY(x))_{x \in \rmD}$ satisfying
\begin{equation*}
\E\bigl[\rmX(x)\rmX(y)\bigr] \leq \E\bigl[\rmY(x)\rmY(y)\bigr] \,, \qquad \forall \, x, y \in \rmD \;.
\end{equation*} 
Let $\varphi : \R^{+} \to \mathbb{R}$ be a convex function with at most polynomial growth at $0$ and $\infty$, and let $f: \rmD \to [0, \infty)$ be a continuous function. Then, we have that
\begin{equation*}
\E\Biggl[\varphi\Biggl(\int_{\rmD} f(x) e^{\rmX(x)-\frac{1}{2}\E[\rmX(x)^2]} d x \Biggr)\Biggr] \leq \E\Biggl[\varphi\Biggl(\int_{\rmD} f(x) e^{\rmY(x)-\frac{1}{2}\E[\rmY(x)^2]} d x \Biggr)\Biggr] \;.
\end{equation*}
\end{lemma}

The remainder of this section is organised as follows. In Section~\ref{sub:compactFields}, we introduce suitable collections of fields with finite-range correlations that approximate $\rmX^{\star}$ and $\rmW$. Then, in Section~\ref{sub:remCutOff}, we show how to remove the cutoff and obtain the desired result using Kahane's convexity inequality.

\subsection{A collection of fields with finite-range correlations}
\label{sub:compactFields}
We begin by introducing the fields that will be the focus of this section. To this end, we consider a non-negative, radial function $\varphi \in \CC_c^{\infty}(\R^d)$ whose support is contained in $B(0, 1/2)$ and satisfies $\int_{\R^d} \varphi(x)^2 dx = 1$. We then define the function $\chi: \R^d \to \R$ by setting
\begin{equation*}
		\chi(x) \eqdef (\varphi * \varphi) (x) \;, \qquad \forall \, x \in \R^d \;.
\end{equation*}
Under the above assumptions on $\varphi$, it is easily seen that $\chi \in \CC_c^{\infty}(\R^d)$ is non-negative, radial, with support contained in $B(0,1)$, and such that $\chi(0) = 1$. Furthermore, for $\updelta > 0$, we define $\smash{\chi_{\delta}(\cdot) \eqdef \chi(\delta \cdot)}$, so that $\chi_{\delta}$ converges to $1$ over compact subsets of $\R^d$ as $\delta \to 0$. With this notation in hand, we introduce the following collections of fields.
\begin{itemize}
\item Let $\frkK$ be the seed covariance function of the field $\rmX^{\star}$ defined in \eqref{eq:field}. For every $\delta > 0$, we define the function $\frkK_{\delta} :\R^d \to \R$ by letting, for all $x \in \R^d$,
\begin{equation*}
\frkK_{\delta}(x) \eqdef \frkK(x) \chi_{\delta}(x) \;,
\end{equation*}
which is a positive definite kernel as it is the product of two positive definite functions. We let $\bar \frkK_{\delta}:\R^d \to \R$ be the function such that the convolution with itself equals $\frkK_{\delta}$.  
Let $\xi$ be the space-time white noise on $\R^d \times \R^{+}$ used in the definition \eqref{eq:field} of $\rmX^{\star}$. We define the field $\rmX^{\smash[b]{\star, \delta}}$ by letting 
\begin{equation}
\label{eq:fieldDelta}
\rmX^{\smash[b]{\star, \delta}}(\cdot) \eqdef \int_{\R^d} \int_0^{\infty} \bar\frkK_{\delta}\bigl(e^{r}(y - \cdot)\bigr) e^{\f{dr}{2}} \xi(dy, dr) \;.
\end{equation}
Furthermore, for $0 \leq s < t$, we let $\rmX^{\smash[b]{\star, \delta}}_{s, t}$ be the field on $\R^d$ given by
\begin{equation*}
\rmX^{\smash[b]{\star, \delta}}_{s, t}(\cdot) \eqdef \int_{\R^d} \int_{s}^{t} \bar\frkK_{\delta}\bigl(e^{r}(y -\cdot)\bigr) e^{\f{dr}{2}} \xi(dy, dr) \;,
\end{equation*}
with the convention that the subscript $s$ is dropped when $s = 0$.
By definition, the seed covariance function $\frkK_{\delta}$ satisfies all the conditions stated in \ref{hp_K1}\dash\ref{hp_K2} with the further property that $\frkK_{\delta}$ has compact support. For all $0 \leq s < t$, the field $\rmX^{\smash[b]{\star, \delta}}_{s,t}$ has the following covariance structure,
\begin{equation}
\label{eq:covsDelta}
\E\bigl[\rmX^{\smash[b]{\star, \delta}}_{s, t}(x) \rmX^{\smash[b]{\star, \delta}}_{s, t}(y)\bigr] = \int_s^{t} \frkK_{\delta}\bigl(e^{r} (x-y)\bigr) dr \;, \qquad \forall \, x, y \in \R^d \;.
\end{equation}
For $0 \leq s < t$, we also introduce the field $\rmX^{\star, \delta, s}_t$ (note that $s$ appears in the superscript), defined by letting
\begin{equation}
\label{eq:fieldSTHigh}
\rmX^{\star, s, \delta}_t(\cdot) \eqdef \rmX^{\smash[b]{\star}}_s(\cdot) + \rmX^{\star, \delta}_{s, t}(\cdot) \;.
\end{equation}
We observe that, by construction, the two fields on the right-hand side of the above display are independent. 
\item For each $\delta > 0$, we define the field $\rmW_{\delta}$ as the stationary, centred Gaussian field with covariance kernel given by
\begin{equation*}
\E\bigl[\rmW_{\delta}(x) \rmW_{\delta}(y)\bigr] =  \E\bigl[\rmW(x) \rmW(y)\bigr] \chi_{\delta}(x-y) \;, \qquad \forall \, x, y \in \R^d \;,
\end{equation*} 
which is still a valid covariance kernel as it is the product of two positive definite functions. Furthermore, for every $t \geq 0$, we define $\smash{\rmW^{\delta}_{t}}$ as the stationary, centred Gaussian field with covariance kernel given by 
\begin{equation*}
\E\bigl[\rmW^{\delta}_{t}(x) \rmW^{\delta}_{t}(y)\bigr] = \E\bigl[\rmW^{\delta}(e^t x) \rmW^{\delta}(e^t y)\bigr] \;, \qquad \forall \, x, y \in \R^d \;.
\end{equation*}
\end{itemize}

We now state and prove the following result which guarantees that $\rmX^{\smash[b]{\star,\delta}}_{s, t}$ and $\rmW^{\delta}_{t}$ are ``good approximation'' of $\rmX^{\star}_{s, t}$ and $\rmW_{t}$, respectively.
\begin{lemma}
\label{lm:covIneq}
For any $\eps>0$ there exists $\delta > 0$ small enough such that, for all $x$, $y \in\R^d$ and $0 \leq s < t$, it holds that
\begin{align}
\E\bigl[\rmX^{\star, \delta}_{s, t}(x)\rmX^{\star, \delta}_{s, t}(y)\bigr] - \eps & \leq \E\bigl[\rmX^{\star}_{s, t}(x)\rmX^{\star}_{s, t}(y)\bigr] \leq \E\bigl[\rmX^{\star, \delta}_{s, t}(x)\rmX^{\star, \delta}_{s, t}(y)\bigr] + \eps \;, \label{eq:ineqCovSuff1}\\ 
\E\bigl[\rmW_t^{\delta}(x)\rmW_t^{\delta}(y)\bigr] - \eps & \leq \E\bigl[\rmW_t(x)\rmW_t(y)\bigr] \leq \E\bigl[\rmW_t^{\delta}(x)\rmW_t^{\delta}(y)\bigr] + \eps \label{eq:ineqCovSuff2}\;.
\end{align}
\end{lemma}
\begin{proof}
Since all the fields involved are stationary, we can, without loss of generality, set $y = 0$ for simplicity.
We start by proving \eqref{eq:ineqCovSuff1}. Let $\eps > 0$, $x \in\R^d$, and $0 \leq s < t$. Recalling \eqref{eq:covs} and \eqref{eq:covsDelta}, we note that 
\begin{equation*}
\abs{\E\bigl[\rmX^{\star}_{s, t}(x)\rmX^{\star}_{s, t}(0)\bigr] - \E\bigl[\rmX^{\star, \delta}_{s, t}(x)\rmX^{\star, \delta}_{s, t}(0)\bigr] } 
\leq \int_0^{\infty} \abs{\frkK\bigl(e^{r} x\bigr) - \frkK_{\delta}\bigl(e^{r} x\bigr)} dr \;.
\end{equation*}
At this point, the conclusion follows by applying the exact same strategy as in the proof of \cite[Lemma~6.2]{Glassy}.

We now prove \eqref{eq:ineqCovSuff2}. Let $\eps > 0$, $x \in\R^d$, and $0 \leq s < t$. Thanks to the decay of correlations of the field $\rmW$, we can find $\rmR = \rmR(\eps) >0$ large enough, only depending on $\eps$, such that $\abs{\E[\rmW(x)\rmW(0)]} \leq \eps$, for all $|x| > \rmR$. Hence, for all $t \geq 0$,
\begin{equation*}
\sup_{\abs{x} > \rmR} \bigl|\E\bigl[\rmW_t(x)\rmW_t(0)\bigr] - \E\bigl[\rmW_t^{\delta}(x)\rmW_t^{\delta}(0)\bigr]\bigr| = \sup_{\abs{x} > \rmR} \bigl|\E\bigl[\rmW_t(x)\rmW_t(0)\bigr]\bigr| \abs{1 - \chi_{\delta}(x)} \leq \eps.
\end{equation*}
On the other hand, uniformly over $|x| \leq \rmR$, it holds that $\chi_{\delta}(x) \to 1$ as $\delta \to 0$. Therefore, in this case, the conclusion follows from the fact that the covariance kernel of $\rmW$ is uniformly bounded on $\R^d \times \R^d$.
\end{proof}

\subsection{Removing the cutoff}
\label{sub:remCutOff}
For each $\delta > 0$ and $0 \leq s < t$, we define the measure $\mu^{\smash[b]{\star, s, \delta}}_{\gamma, t}$ on $\R^d$ by letting
\begin{equation}
\label{e:seq_super_main_star_delta}
\mu^{\smash[b]{\star, s, \delta}}_{\gamma, t}(dx) = t^{\f{3\gamma}{2\sqrt{\smash[b]{2d}}}} e^{t(\gamma/\sqrt 2-\sqrt{d})^2} e^{\gamma (\rmX^{\smash[b]{\star, s, \delta}}_{t}(x) + \rmW^{\delta}_{t}(x)) - \frac{\gamma^2}{2} \E[{\rmX^{\smash[b]{\star, s, \delta}}_{t}(x)}^2]} dx\;,
\end{equation}
where we recall that the field $\smash{\rmX^{\smash[b]{\star, s, \delta}}_{t}}$ is defined in \eqref{eq:fieldSTHigh}. Furthermore, we recall that $\mu^{\smash[b]{\star}}_{\gammac}$ denotes the critical GMC measure associated with $\rmX^{\star}$. We have the following key result which is a consequence of \cite{BHCluster}. 

\begin{lemma}
\label{lm:StableConvFixedDelta}
 For $\smash{\gamma > \sqrt{\smash[b]{2d}}}$ and $\delta > 0$, consider the sequence of random measures $\smash{(\mu^{\smash[b]{\star, s,\delta}}_{\gamma, t})_{0 \leq s < t}}$ introduced in \eqref{e:seq_super_main_star_delta}. Then, there exists a finite constant $c^{\star}_{\gamma, \delta} > 0$ such that $\smash{\mu^{\smash[b]{\star, s,\delta}}_{\gamma, t}}$ converges $\sigma(\rmX^{\star})$-stably to $\smash{c^{\star}_{\gamma, \delta} \CP_{\gamma}[\mu^{\smash[b]{\star}}_{\gammac}]}$ as $t \to \infty$, followed by $s \to \infty$. 
\end{lemma}
\begin{proof}
This is a consequence of the proof of \cite[Theorem~C]{BHCluster}. To be precise, we recall that the field $\smash{\rmX^{\smash[b]{\star, s, \delta}}_{t}}$ is given by the following sum
\begin{equation*}
\rmX^{\star, s, \delta}_t = \rmX^{\smash[b]{\star}}_s + \rmX^{\star, \delta}_{s, t} \;,
\end{equation*}
where the first term on the right-hand side corresponds to the ``large scales field'' and the second term to the ``small scales field''. The only subtlety is that the large scales field does not have compactly supported covariance. However, the proof of \cite[Theorem~C]{BHCluster} (see \cite[Section~6]{BHCluster} for more details) remains valid in this more general setting. Indeed, at no point in the proof do we rely on the large scales field having compactly supported covariance, whereas it is crucial that the small scales field does.
\end{proof}

In what follows, in order to lighten the notation, we let 
\begin{equation}
\label{eq:notationIntroLight}
\mu^{\smash[b]{\star, \delta}}_{\gamma} = c^{\star}_{\gamma, \delta} \CP_{\gamma}[\mu^{\smash[b]{\star}}_{\gammac}] \;,  \qquad \mu^{\star}_{\gamma} = c^{\star}_{\gamma} \CP_{\gamma}[\mu^{\star}_{\gammac}] \;, 
\end{equation}
for some constant $c^{\star}_{\gamma}$ (that will be identified with the limit as $\delta \to 0$ of $c^{\star}_{\gamma, \delta}$). 

Now, thanks to Lemma~\ref{lm:stableRV} and \cite[Lemma~3.4]{BHCluster}, the convergence in Lemma~\ref{lm:StableConvFixedDelta} holds if and only if, for all $(\varphi, f) \in \CC_c(\R^d) \times \CC_{c}^{+}(\R^d)$, 
\begin{equation}
\label{eq:consCluster}
\lim_{s \to \infty} \lim_{t \to \infty} \E\Bigl[\exp\bigl(i \langle \rmX^{\star}, \varphi \rangle\bigr) \exp\bigl(-\mu^{\smash[b]{\star, s, \delta}}_{\gamma, t}(f) \bigr)\Bigr]  = \E \Bigl[\exp\bigl(i \langle \rmX^{\star}, \varphi \rangle\bigr) \exp\bigl(-\mu^{\smash[b]{\star, \delta}}_{\gamma}(f) \bigr)\Bigr] \;.
\end{equation}
Therefore, in order to prove Lemma~\ref{lm:stableConvStartn}, we would like to replace $\smash{\mu^{\smash[b]{\star, s, \delta}}_{\gamma, t}}$ by $\smash{\mu^{\star}_{\gamma, t}}$ and $\smash{\mu^{\smash[b]{\star, \delta}}_{\gamma}}$ by $\smash{\mu^{\star}_{\gamma}}$ in both sides of the above convergence. To do so, we divide the proof in two main lemmas. Specifically, in Lemma~\ref{lem:FirstRemComp}, we show that the limit of $\mu_{\gamma, t}^{\star}$ as $t \to \infty$ coincides with the limit of $\smash{\mu_{\gamma}^{\star, \delta}}$ as $\delta \to 0$. Then, in Lemma~\ref{lem:SecondRemComp}, we compute the limit of $\smash{\mu_{\gamma}^{\star, \delta}}$ as $\delta \to 0$. 
\begin{lemma}
\label{lem:FirstRemComp}
For any $\gamma > \sqrt{\smash[b]{2d}}$ and for any $(\varphi, f) \in \CC_c(\R^d) \times \CC_{c}^{+}(\R^d)$, it holds that
\begin{equation*}
\lim_{t \to \infty} \E\Bigl[\exp\bigl(i \langle\rmX^{\star}, \varphi\rangle\bigr)\exp\bigl(-\mu^{\star}_{\gamma, t}(f) \bigr)\Bigr]  = \lim_{\delta \to 0} \E \Bigl[\exp\bigl(i \langle\rmX^{\star}, \varphi\rangle\bigr)\exp\bigl(-\mu^{\smash[b]{\star, \delta}}_{\gamma}(f) \bigr)\Bigr] \;,
\end{equation*}
where both limits exist and are nontrivial, and where we recall that $\mu^{\star}_{\gamma, t}$ is defined in \eqref{e:seq_super_main_star_ttt}, and $\mu_{\gamma}^{\smash[b]{\star, \delta}}$ in \eqref{eq:notationIntroLight}.
\end{lemma}
\begin{proof}
We fix $(\varphi, f) \in \CC_c(\R^d) \times \CC_{c}^{+}(\R^d)$. For a field $(\rmY(x))_{x \in \R^d}$, we let 
\begin{equation*}
\begin{alignedat}{2}
& \mathrm{Re}_{+}(\rmY, \varphi) \eqdef  \bigl[\mathrm{Re}\bigl(\exp\bigl(i \langle\rmY,\varphi\rangle\bigr)\bigr)\bigr
]^{+}\;, \qquad 
&& \mathrm{Re}_{-}(\rmY, \varphi) \eqdef \bigl[\mathrm{Re}\bigl(\exp\bigl(i \langle\rmY,\varphi\rangle\bigr)\bigr)\bigr
]^{-}\;, \\
& \mathrm{Im}_{+}(\rmY, \varphi) \eqdef  \bigl[\mathrm{Im}\bigl(\exp\bigl(i \langle\rmY,\varphi\rangle\bigr)\bigr)\bigr]^{+}\;, \qquad
&&  \mathrm{Im}_{-}(\rmY, \varphi) \eqdef  \bigl[\mathrm{Im}\bigl(\exp\bigl(i \langle\rmY,\varphi\rangle\bigr)\bigr)\bigr]^{-} \;,
\end{alignedat}
\end{equation*}
where $x^{+} = x \vee 0$ and $x^{-} = (-x) \vee 0$.

For $0 \leq s < t$ and $\delta > 0$, by definition of the measure $\smash{\mu^{\smash[b]{\star, s, \delta}}_{\gamma, t}}$ in \eqref{e:seq_super_main_star_delta} and the field $\smash{\rmX^{\smash[b]{\star, s, \delta}}_{t}}$ in \eqref{eq:fieldSTHigh}, we have that
\begin{equation*}
\mu^{\smash[b]{\star, s, \delta}}_{\gamma, t}(f) = t^{\f{3\gamma}{2\sqrt{\smash[b]{2d}}}}e^{t(\gamma/\sqrt{2} - \sqrt{d})^2} \int_{\R^d} f(x) e^{\gamma (\rmX^{\smash[b]{\star}}_s(x) + \rmX^{\star, \delta}_{s, t}(x) + \rmW_t^{\delta}(x)) -  \frac{\gamma^2}{2} \E[{\rmX^{\smash[b]{\star}}_{s}(x)}^2 + {\rmX^{\smash[b]{\star, \delta}}_{s, t}(x)}^2]} dx \;, 
\end{equation*}
where we recall that the field $\smash{\rmX^{\smash[b]{\star}}_{s}}$ and $\smash{\rmX^{\smash[b]{\star, \delta}}_{s, t}}$ are independent. 
We observe that, for any $\eps> 0$, thanks to Lemma~\ref{lm:covIneq}, for $\delta > 0$ small enough and for all $x$, $y \in \R^d$, it holds that
\begin{equation*}
\E\bigl[\rmX^{\star, \delta}_{s, t}(x)\rmX^{\star, \delta}_{s, t}(y)\bigr] + \E\bigl[\rmW_t^{\delta}(x)\rmW_t^{\delta}(y)\bigr] \leq \E\bigl[\rmX^{\star}_{s, t}(x)\rmX^{\star}_{s, t}(y)\bigr] + \E\bigl[\rmW_t(x)\rmW_t(y)\bigr] + \eps \;,
\end{equation*}
where we recall that the fields $\smash{\rmX^{\star, \delta}_{s, t}}$ and $\rmW_t^{\delta}$ are independent. 
Therefore, letting $\CN$ be a standard normal random variable independent of everything else, for $0 \leq u < s < t$, we can apply Kahane's convexity inequality (Lemma~\ref{lm_Kahane}) with respect to the conditional expectation given $\CF_s$ (recall definition \eqref{eq:filtrationStar}) with the function $\R^{+} \ni x \mapsto e^{-x}$ to obtain that
\begin{equation*}
\E\Bigl[\mathrm{Re}_{+}(\rmX^{\star}_u, \varphi) \exp\bigl(-\mu^{\smash[b]{\star, s, \delta}}_{\gamma, t}(f) \bigr)\Bigr]  \leq	\E\Bigl[\mathrm{Re}_{+}(\rmX^{\star}_u, \varphi) \exp\bigl(- e^{\gamma \sqrt{\eps} \CN - \frac{\gamma^2}{2}\eps} \mu^{\star}_{\gamma, t}(f) \bigr)\Bigr] \;.
\end{equation*}
In particular, from the above inequality, we can deduce that 
\begin{equation}
\label{eq:Kahane1}
\begin{alignedat}{1}
\E\Bigl[\mathrm{Re}_{+}(\rmX^{\star}, \varphi)\exp\bigl(-\mu^{\smash[b]{\star, s, \delta}}_{\gamma, t}(f) \bigr)\Bigr]  
& \leq \E\Bigl[\mathrm{Re}_{+}(\rmX^{\star}, \varphi) \exp\bigl(- e^{\gamma \sqrt{\eps} \CN - \frac{\gamma^2}{2}\eps} \mu^{\star}_{\gamma, t}(f) \bigr)\Bigr] \\
& + 2 \E\bigl[\bigl|\mathrm{Re}_{+}(\rmX^{\star}, \varphi) - \mathrm{Re}_{+}(\rmX^{\star}_u, \varphi)\bigr|\bigr]\;.
\end{alignedat}
\end{equation}

Similarly, for any $\eps> 0$, thanks again to Lemma~\ref{lm:covIneq}, we know that for $\delta > 0$ small enough and for all $x$, $y \in \R^d$, it holds that 
\begin{equation*}
\E\bigl[\rmX^{\star}_{s, t}(x)\rmX^{\star}_{s, t}(y)\bigr] + \E\bigl[\rmW_t(x)\rmW_t(y)\bigr] \leq \E\bigl[\rmX^{\star, \delta}_{s, t}(x)\rmX^{\star, \delta}_{s, t}(y)\bigr] + \E\bigl[\rmW_t^{\delta}(x)\rmW_t^{\delta}(y)\bigr] + \eps \;. 
\end{equation*}
Hence, letting $\CN$ be as above, and by applying Kahane's convexity inequality (Lemma~\ref{lm_Kahane}) as above, we obtain that
\begin{equation}
\label{eq:Kahane2}
\begin{alignedat}{1}
\E\Bigl[\mathrm{Re}_{+}(\rmX^{\star}, \varphi) \exp\bigl(-\mu^{\star}_{\gamma, t}(f) \bigr)\Bigr]  
& \leq \E\Bigl[\mathrm{Re}_{+}(\rmX^{\star}, \varphi) \exp\bigl(- e^{\gamma \sqrt{\eps} \CN - \frac{\gamma^2}{2}\eps} \mu^{\smash[b]{\star, s, \delta}}_{\gamma, t}(f) \bigr)\Bigr] \\
& + 2 \E\bigl[\bigl|\mathrm{Re}_{+}(\rmX^{\star}, \varphi) - \mathrm{Re}_{+}(\rmX^{\star}_u, \varphi)\bigr|\bigr]\;.
\end{alignedat}
\end{equation}

Now, we observe that the sequence of random measures $\smash{(\mu^{\star}_{\gamma, t})_{t \geq 0}}$ is tight under the topology of vague convergence, and every converging subsequence is nontrivial (see \cite[Proposition~10]{Critical_der} whose proof is not reliant on the seed covariance function being compactly supported).
In particular, by taking the limit as $t' \to \infty$ along a subsequence in \eqref{eq:Kahane1}, for any $\zeta \in (0, 1)$, we obtain that
\begin{align}
\label{eq:techKCI1}
\lim_{t' \to \infty} \E\Bigl[\mathrm{Re}_{+}(\rmX^{\star}, \varphi) & \exp\bigl(- (1 - \zeta)\mu^{\star}_{\gamma, t'}(f) \bigr)\Bigr] + \P\Bigl(e^{\gamma \sqrt{\eps} \CN - \frac{\gamma^2}{2}\eps} \leq 1 - \zeta\Bigr) \nonumber\\
& \geq \lim_{t' \to \infty} \E \Bigl[\mathrm{Re}_{+}(\rmX^{\star}, \varphi) \exp\bigl(- e^{\gamma \sqrt{\eps} \CN - \frac{\gamma^2}{2}\eps} \mu^{\star}_{\gamma, t'}(f) \bigr)\Bigr] \nonumber \\
& \geq \E\Bigl[\mathrm{Re}_{+}(\rmX^{\star}, \varphi) \exp\bigl(-\mu^{\smash[b]{\star, \delta}}_{\gamma}(f)\bigr)\Bigr] \;,
\end{align} 
where we also took the limits as $s \to \infty$ and $u \to \infty$ in \eqref{eq:Kahane1} and we applied Lemma~\ref{lm:StableConvFixedDelta} along with the fact that $\rmX_{u}$ converges almost surely to $\rmX$ in $\CH^{-\kappa}(\R^d)$ for any $\kappa > 0$.
 
Proceeding in the same manner, by taking the limit as $t' \to \infty$ along the same subsequence in \eqref{eq:Kahane2}, for any $\zeta \in (0, 1)$, we get that
\begin{align}
\lim_{t' \to \infty} \E\Bigl[\mathrm{Re}_{+}(\rmX^{\star}, \varphi) & \exp\bigl(-\mu^{\star}_{\gamma, t'}(f) \bigr)\Bigr] \nonumber\\
&  \leq \E \Bigl[\mathrm{Re}_{+}(\rmX^{\star}, \varphi) \exp\bigl(-e^{\gamma \sqrt{\eps} \CN - \frac{\gamma^2}{2}\eps} \mu^{\smash[b]{\star, \delta}}_{\gamma}(f)\bigr) \Bigr] \nonumber \\ 
& \leq \E \Bigl[\mathrm{Re}_{+}(\rmX^{\star}, \varphi)\exp\bigl( -(1 - \zeta) \mu^{\smash[b]{\star, \delta}}_{\gamma}(f) \bigr)\Bigr] + \P\Bigl(e^{\gamma \sqrt{\eps} \CN - \frac{\gamma^2}{2}\eps} \leq 1 - \zeta\Bigr) \;. \label{eq:techKCI2}
\end{align} 
In particular, by taking the $\limsup_{\delta \to 0}$ and the $\liminf_{\delta \to 0}$ followed by the $\lim_{\eps \to 0}$ in \eqref{eq:techKCI1} and in \eqref{eq:techKCI2}, we obtain that
\begin{align*}
\limsup_{\delta \to 0} \E \Bigl[\mathrm{Re}_{+}(\rmX^{\star}, \varphi) \exp\bigl( -\mu^{\smash[b]{\star, \delta}}_{\gamma}(f) \bigr)\Bigr] & \leq \lim_{t' \to \infty} \E\Bigl[\mathrm{Re}_{+}(\rmX^{\star}, \varphi) \exp\bigl(-(1 - \zeta) \mu^{\star}_{\gamma, t'}(f) \bigr)\Bigr] \;, \\
\liminf_{\delta \to 0} \E\Bigl[\mathrm{Re}_{+}(\rmX^{\star}, \varphi) \exp\bigl(-\mu^{\smash[b]{\star, \delta}}_{\gamma}(f) \bigr)\Bigr] & \geq \lim_{t' \to \infty} \E\Bigl[\mathrm{Re}_{+}(\rmX^{\star}, \varphi) \exp\bigl(-(1 - \zeta)^{-1} \mu^{\star}_{\gamma, t'}(f) \bigr)\Bigr] \;.
\end{align*}
Therefore, by arbitrariness of $\zeta \in (0, 1)$, we get that
\begin{equation*}
\lim_{t' \to \infty} \E\Bigl[\mathrm{Re}_{+}(\rmX^{\star}, \varphi) \exp\bigl(- \mu^{\star}_{\gamma, t'}(f) \bigr)\Bigr] = \lim_{\delta \to 0} \E \Bigl[\mathrm{Re}_{+}(\rmX^{\star}, \varphi) \exp\bigl( -\mu^{\smash[b]{\star, \delta}}_{\gamma}(f)\bigr) \Bigr] \;.
\end{equation*}
By applying the same argument to $\mathrm{Re}_{-}(\rmX^{\star}, \varphi)$, $\mathrm{Im}_{+}(\rmX^{\star}, \varphi)$, and $\mathrm{Im}_{-}(\rmX^{\star}, \varphi)$, the desired result follows from the fact that the subsequence along which the limit was taken was arbitrary.
\end{proof}

\begin{lemma}
\label{lem:SecondRemComp}
For any $\gamma > \sqrt{\smash[b]{2d}}$, there exists a finite constant $a_{\star} > 0$ such that, for any $(\varphi, f) \in \CC_c(\R^d) \times \CC_{c}^{+}(\R^d)$, it holds that
\begin{equation*}
\lim_{\delta \to 0} \E\Bigl[\exp\bigl(i \langle\rmX^{\star}, \varphi\rangle\bigr) \exp\bigl(- \mu^{\smash[b]{\star, \delta}}_{\gamma}(f) \bigr)\Bigr] = \E\Bigl[\exp\bigl(i \langle\rmX^{\star}, \varphi\rangle\bigr) \exp\bigl(- \mu^{\star}_{\gamma}(f) \bigr)\Bigr] \;,
\end{equation*}
where we recall that the constant $c^{\star}_{\gamma}$ appears in the definition \eqref{eq:notationIntroLight} of the measure $\mu_{\gamma}^{\star}$.
\end{lemma}
\begin{proof}
Recalling \eqref{eq:notationIntroLight}, we note that it suffices to prove that $\lim_{\delta \to 0} c^{\star}_{\gamma, \delta} = c^{\star}_{\gamma}$ for some positive $c^{\star}_{\gamma}> 0$. This fact follows immediately from Lemma~\ref{lem:FirstRemComp}.
\end{proof}

We are finally ready to prove Lemma~\ref{lm:stableConvStartn}, whose proof follows immediately by combining Lemmas~\ref{lem:FirstRemComp} and~\ref{lem:SecondRemComp}.
\begin{proof}[Proof of Lemma~\ref{lm:stableConvStartn}]
We recall that it suffices to prove that for all $(\varphi, f) \in \CC_c(\R^d) \times \CC_{c}^{+}(\R^d)$, 
\begin{equation}
\label{eq:consClusterFin}
\lim_{t \to \infty} \E\Bigl[\exp\bigl(i \langle \rmX^{\star}, \varphi \rangle\bigr) \exp\bigl(-\mu^{\star}_{\gamma, t}(f) \bigr)\Bigr]  = \E \Bigl[\exp\bigl(i \langle \rmX^{\star}, \varphi \rangle\bigr) \exp\bigl(-\mu^{\star}_{\gamma}(f) \bigr)\Bigr] \;.
\end{equation}
By using the triangle inequality, for any $t \geq 0$ and $\delta > 0$, we have that 
\begin{align*}
\Bigl| \E\Bigl[\exp\bigl(i & \langle \rmX^{\star}, \varphi \rangle\bigr)\exp\bigl(-\mu^{\star}_{\gamma, t}(f) \bigr)\Bigr]- \E \Bigl[\exp\bigl(i \langle \rmX^{\star}, \varphi \rangle\bigr) \exp\bigl(-\mu^{\star}_{\gamma}(f) \bigr)\Bigr]\Bigr| \\
& \leq  \Bigl| \E\Bigl[\exp\bigl(i \langle \rmX^{\star}, \varphi \rangle\bigr) \exp\bigl(-\mu^{\star}_{\gamma, t}(f) \bigr)\Bigr] - \E \Bigl[\exp\bigl(i \langle \rmX^{\star}, \varphi \rangle\bigr) \exp\bigl(-\mu^{\star, \delta}_{\gamma}(f) \bigr)\Bigr]\Bigr| \\
& + \Bigl| \E\Bigl[\exp\bigl(i \langle \rmX^{\star}, \varphi \rangle\bigr) \exp\bigl(-\mu^{\star, \delta}_{\gamma}(f) \bigr)\Bigr] - \E \Bigl[\exp\bigl(i \langle \rmX^{\star}, \varphi \rangle\bigr) \exp\bigl(-\mu^{\star}_{\gamma}(f) \bigr)\Bigr]\Bigr| \;.
\end{align*}
Therefore, the conclusion follows by a direct application of Lemmas~\ref{lem:FirstRemComp}~and~\ref{lem:SecondRemComp}.
\end{proof}

\appendix

\section{Moments of supercritical GMC}
\label{sec:moments}
In this appendix, we gather some properties concerning the existence of moments and the multifractal spectrum of supercritical GMC measures, which follow directly from their definition. We emphasise that the results presented in this appendix are not used anywhere in the present paper but are recorded here for future reference.

In what follows, we consider the same setting as specified in Theorem~\ref{th:convergence}. In particular, for $\gamma > \sqrt{\smash[b]{2d}}$, we let $\mu_{\gamma}$ denote the supercritical GMC, i.e.,
\begin{equation*}
\mu_{\gamma} = a_{\gamma, \rho} \CP_{\gamma}[\bar{\mu}_{\gammac}] \;,
\end{equation*}
where we recall that $\smash{\bar{\mu}_{\gammac}}$ is the critical GMC defined in \eqref{eq:tiltedCritical}.

\begin{proposition}[Positive moments]
For $\gamma>\sqrt{\smash[b]{2d}}$ and for any $\rmA \subset \rmD$ non-empty, bounded and open, the random variable $\mu_{\gamma}(\rmA)$ posses finite moments of order $q  \in (0, \sqrt{\smash[b]{2d}}/\gamma)$. 	
\end{proposition}
\begin{proof}
Let $\gamma > \sqrt{\smash[b]{2d}}$ and fix a set $\rmA \subset \rmD$ as in the proposition statement. We note that for every $x \geq 0$ and $q \in (0,1)$ it holds that
\begin{equation}
\label{e:rel1}
x^{q} = - \frac{1}{\Gamma(-q)} \int_0^{\infty} \bigl(1 - \exp(-z x)\bigr) \f{dz}{z^{1 + q}}  \;.
\end{equation}
Hence, for $q  \in (0, \sqrt{\smash[b]{2d}}/\gamma)$, thanks to \eqref{e:rel1} and \eqref{eq:comp_Laplace}, there exists a constant $c > 0$ such that
\begin{align*}
\E\bigl[\mu_{\gamma}(\rmA)^q\bigr] & = - \frac{1}{\Gamma(-q)} \int_0^{\infty} \bigl(1 - \E\bigl[\exp\bigl(-z \mu_{\gamma}(\rmA)\bigr)\bigr]\bigr) \f{dz}{z^{1 + q}} \\
	& = - \frac{1}{\Gamma(-q)} \int_0^{\infty} \bigl(1 - \E\bigl[\exp\bigl(- c z^{\sqrt{\smash[b]{2d}}/\gamma} \bar \mu_{\gammac}(\rmA)\bigr)\bigr]\bigr) \f{dz}{z^{1 + q}} \;.
\end{align*}
Therefore, performing a change of variables, one obtains that
\begin{align*}
	\E\bigl[\mu_{\gamma}(\rmA)^q\bigr] 
= \frac{\gamma c^{\f{\gamma q}{\sqrt{\smash[b]{2d}}}} \Gamma\bigl(-\gamma q / \sqrt{\smash[b]{2d}}\bigr)} {\sqrt{\smash[b]{2d}} \Gamma(-q)} \E\bigl[\bar{\mu}_{\gammac}(\rmA)^{\gamma q/\sqrt{\smash[b]{2d}}}\bigr]  \;,
\end{align*}
and the result now follows thanks to the fact that $\gamma q/\sqrt{\smash[b]{2d}} < 1$ and from \cite[Theorem~2.11]{Powell_Critical}. 
\end{proof}

\begin{proposition}[Negative moments]
For $\gamma > \sqrt{\smash[b]{2d}}$ any for any $\rmA \subset \rmD$ non-empty, bounded and open, the random variable $\mu_{\gamma}(\rmA)$ posses finite moments of every order $q < 0$. 	
\end{proposition}
\begin{proof}
Let $\gamma > \sqrt{\smash[b]{2d}}$ and fix a set $\rmA \subset \rmD$ as in the proposition statement. We recall that for every $x > 0$ and $q > 0$ it holds that
\begin{equation}
\label{e:rel2}
	\Gamma(q) = x^q \int_0^{\infty} \exp(-zx) \f{dz}{z^{1-q}} \;.
\end{equation}
Hence, if we fix $q > 0$, thanks to \eqref{e:rel2} and \eqref{eq:comp_Laplace}, there exists a constant $c > 0$ such that
\begin{align*}
\E\bigl[\mu_{\gamma}(\rmA)^{-q}\bigr] & = \frac{1}{\Gamma(q)}\int_0^{\infty} \E\bigl[\exp\left(-z \mu_{\gamma}(\rmA)\right)\bigr] \f{dz}{z^{1-q}} \\
	& = \frac{1}{\Gamma(q)} \int_0^{\infty} \E\Bigl[\exp\bigl(- c z^{\sqrt{\smash[b]{2d}}/\gamma} \bar{\mu}_{\gammac}(\rmA)\bigr)\Bigr] \f{dz}{z^{1 - q}} \;.
\end{align*}
Therefore, performing a change of variables, one obtains that
\begin{align*}
\E\bigl[\mu_{\gamma}(\rmA)^{-q}\bigr]  
= \frac{\gamma \Gamma(\gamma q / \sqrt{\smash[b]{2d}})}{\sqrt{\smash[b]{2d}} c^{\f{\gamma q}{\sqrt{\smash[b]{2d}}}} \Gamma(q)} \E\bigl[\bar{\mu}_{\gammac}(\rmA)^{-\gamma q/\sqrt{\smash[b]{2d}}}\bigr]  \;.
\end{align*}
and the result now follows from \cite[Theorem~2.11]{Powell_Critical}. 
\end{proof}

\begin{proposition}[Multifractal spectrum]
For $\gamma > \sqrt{\smash[b]{2d}}$ and for any $q < \sqrt{\smash[b]{2d}}/\gamma$, there exists a constant $c_q > 0$ such that for any $\rmA \subset \rmD$ non-empty, bounded and open, it holds that
\begin{equation*}
	\E\bigl[\mu_{\gamma}(r \rmA)^q\bigr] \stackrel{r \to 0}{\asymp} c_q r^{\xi_{\gamma}(q)} \;,
\end{equation*}
where $\xi_{\gamma}(q) = \sqrt{\smash[b]{2d}} \gamma q  - \gamma^2 q^2/2$, and the implicit constant depends only on $\gamma$ and $\rmA$. 
\end{proposition}
\begin{proof}
The result follows from the proofs of the previous two propositions and the known multifractal spectrum for critical GMC measures (see \cite[Theorem~2.11]{Powell_Critical}).
\end{proof}
\endappendix

\small
\bibliographystyle{Martin}
\bibliography{./refs}{}
\end{document}